\documentclass[12pt, a4]{amsart}

\usepackage{fancyhdr}
\usepackage[all]{xy}
\usepackage{graphicx}
\usepackage{amsmath}
\usepackage{amsthm}
\usepackage{amssymb}

\usepackage{latexsym,amsfonts,amscd,yhmath,epsfig,epic}

\usepackage{ccaption}
\changecaptionwidth
\captionwidth{5.6in}

\usepackage{imakeidx}
\makeindex            
\newcommand{\del}{{\delta}}
\newcommand{\PP}{{\mathbb{P}}}
\newcommand{\conv}{\operatorname{Conv}}\newcommand{\Pic}{\operatorname{Pic}}
\newcommand{\Iso}{\operatorname{Iso}}
\newcommand{\Int}{\operatorname{Int}}
\newcommand{\Aut}{\operatorname{Aut}}
\newcommand{\R}{\mathbb{R}}
\newcommand{\CC}{{\mathbb C}}  
\renewcommand{\le}{\leq} 

\newcommand{\ko}{\mathcal{O}}
\newcommand{\kf}{\mathcal{F}}
\newcommand{\Sing}{\operatorname{Sing}}
\newcommand{\mt}{\operatorname{mt}}
\newcommand{\bd}{\boldsymbol{d}}
\newcommand{\bq}{\boldsymbol{q}}
\newcommand{\Lam}{\Lambda}
\newcommand{\Sig}{\Sigma}
\newcommand{\eps}{\varepsilon}
\newcommand{\Z}{\mathbb{Z}}

\newtheorem{definition}{Definition}[section]
\newtheorem{theorem}[definition]{Theorem}
\newtheorem{remark}[definition]{Remark}

\newtheorem{conjecture}[definition]{Conjecture}

\newtheorem{proof*}{Proof}
\newtheorem{lemma}[definition]{Lemma}

\makeindex             

\begin{document}

\title{Plane algebraic curves with prescribed singularities}
\author{Gert-Martin Greuel and Eugenii Shustin}
%
%
\maketitle



\begin {abstract} We report on the problem of  the existence of complex and real algebraic curves in the plane with prescribed singularities up to analytic and topological equivalence. The question is whether, for a given positive integer $d$ and a finite number of  given analytic or topological singularity types, there exist a plane (irreducible) curve of degree $d$ having singular points of the given type as its only singularities. The set of all such curves is a quasi-projective variety, which we call an equisingular family $ESF$.
We describe, in terms of numerical invariants of the curves and their singularities, the state of the art concerning necessary and sufficient conditions for the non-emptiness and $T$-smoothness (i.e., smooth of expected dimension) of the corresponding $ESF$.
The considered singularities can be arbitrary, but we spend special attention to plane curves with nodes and cusps, the most studied case, where still no complete answer is known in general.
An important result is, however, that the necessary and the sufficient conditions show the same asymptotics for $T$-smooth equisingular families if the degree goes to infinity.\\
\end{abstract}

{\footnotesize
MSC2020 subject classification: 14-02, 14B, 14D, 14H, 14J,14P, 32-02, 32B, 32C, 32G, 32S

}

\tableofcontents

\section*{Introduction}

\label{sec:1}

%
%
%

Singular algebraic curves, their existence, deformation, families (from the
local and global point of view) attract continuous attention of algebraic
geometers since the last century. The aim of this survey is to give an account
of results, trends and bibliography related to the existence of curves with prescribed singularities with a focus on algebraic curves in the plane.
We consider the existence problem for complex and real plane curves with given singularities up to  analytic and topological equvalence.
The general problem is: given an integer $d>0$ and
analytic or topological singularity types $S_1,\dots,S_r$, does there
exist a curve (resp. an irreducible curve) of degree $d$ in $\PP^2$ having $r$ singular points of
types $S_1,\dots,S_r$, respectively, as its only singularities?

An important particular case is the same problem for one singularity.
Namely, let $S$ be an analytic or topological type.
What is the minimal degree $d(S)$ of a curve in $\PP^2$ having a singular point of
type $S$\,? In other words, we ask about
a polynomial normal form of minimal
degree of the given singularity.

%

The space  $|dH| = |H^0(\PP^2,\ko_{\PP^2}(d))|$ of all curves of degree $d$ in $\PP^2$,
$H$ a hyperplane in $\PP^2$,  can be identified with the punctured vector space of homogeneous polynomials of degree $d$ in $3$ variables modulo multiplication with a non-zero constant.
That is, $|dH| = \CC [x_0,x_1,x_2]_d \smallsetminus \{0\}/ \mathbb C^*$
is a projective space of dimension $N=(d^2+3d)/2$. The subspace of this $\PP^N$, consisting of (irreducible) curves of degree $d$ in $\PP^2$ having $r$ singular points of types $S_1,\dots,S_r$ (and may be other not specified singularities) is the {\em Equisingular Family $(ESF)$} which we denote by
$$V_d^{(irr)} (S_1,\dots,S_r)$$
(it may be empty). This description of $V_d^{(irr)} (S_1,\dots,S_r)$  is set-theoretically, but it is shown in \cite{GLS5}  that these sets  are quasi-projective subvarieties of $\PP^N$ (see \cite[Proposition I 1.61 and Proposition I 1.71]{GLS5} for a simple proof in the case of one singularity), which can be endowed with a unique (not necessarily reduced) scheme structure representing the functor of equianalytic resp. equisingular deformations (see \cite[Theorem II 2.36]{GLS5}).\\

The following geometric problems concerning equisingular families  of plane
curves have been of interest to algebraic geometers since the early 20th century:

\begin{itemize}
\item
{\em Existence Problem}:\index{existence!problem}\index{problem!existence} Is $V_d^{(irr)}(S_1,\dots,S_r)$
non-empty?

\noindent
\item
{\em $T$-Smoothness Problem}:\index{T-smoothness!problem}\index{problem!T-smoothness}
If $V_d^{(irr)}(S_1,\dots,S_r)$ is non-empty, is it $T$-smooth, i.e.
smooth and of the "expected" dimension (see end of the Preliminaries)?


\noindent
\item
{\em Irreducibility Problem}: \index{irreducibility problem}\index{problem!irreducibility} Is  $ V_d^{(irr)}(S_1,\dots,S_r)$ irreducible?

\noindent
\item
{\em Deformation Problem}: \index{deformation problem}\index{problem!deformation}
What are the adjacencies of the singularities of a curve of degree $d$ if it varies inside $|dH|$?
\end{itemize}

First of all, a {\em complete} answer to these questions
is known only for the case of plane nodal curves
(Severi \cite{Se}, Harris \cite{HJ}): the inequality
$0\le n\le\frac{(d-1)(d-2)}{2}$ is necessary and sufficient for
the nonemptiness, $T$-smoothness, and irreducibility of the family  $ V_d^{irr}(nA_1)$
of irreducible plane curves of degree $d$ with $n$ nodes as their
only singularities, and, additionally, for the independent smoothing
of prescribed nodes while keeping the others, induced by the space of plane curves of degree $d$.

Already for plane curves with ordinary cusps a reasonable complete
answer is hardly possible, due to a large gap between the known upper
bounds of the number of cusps and the known examples of curves with many
cusps. Due to the irregular behavior of such examples, it seems unrealistic
to expect a sufficient condition for either non-emptiness,
or $T$-smoothness, or irreducibility, which is at the same time necessary
(as in the case of plane nodal curves).\\

This situation has motivated us to pursue the following goal:
describe the {\em regular region} \index{regular!region} of $V_d^{(irr)}(S_1,\dots,S_r)$ (i.e. the  nonempty and $T$-smooth part), in a possibly precise form, which should be
\begin{enumerate}\item[(i)] {\em universal}, i.e. applicable to arbitrary
singularities, \item[(ii)] {\em numerical}, i.e. expressed as
relations (inequalities) for numerical invariants of the curves and their singularities, \item[(iii)] {\em
asymptotically optimal} or {\em asymptotically proper}, i.e. having
either the same asymptotics or are asymptotically ``comparable" with the known
examples of irregular (empty or non-$T$-smooth) equisingular families if $d$ goes to infinity.\end{enumerate}
We like to emphasize that the one can expect asymptotically optimal or  asymptotically proper results (about
nonemptiness, $T$-smoothness, irreducibility, ...) only for  the regular region, we do not see any systematic behavior for the irregular region of $V_d^{(irr)}(S_1,\dots,S_r)$ if $d \to \infty$.

In this survey we focus mainly on the existence problem and give only a short account on answers to the other problems.
We give always precise references, including original sources and in addition hints to the methods whenever appropriate. We feature both complex and real singular curves. A special attention is paid to curves with nodes and cusps, curves with simple, ordinary, and semi-quasihomogeneous singularities, in which cases one can apply specific constructions and formulate general restrictions in a simpler form.

In general, there is only one universal approach which provides sufficient existence results for arbitrary topological and analytical singularity types and any degree, both over the complex and over the real fields, and which is asymptotically comparable with the necessary conditions. This approach combines two main ingredients:
the {\em theory of zero-dimensional schemes} related to planar curve singularities and the cohomology vanishing theory for their ideal sheaves, and the {\em patchworking construction}\index{patchworking!construction}\index{construction!patchworking}. While the cohomological approach, which builds a bridge between the local and global geometry of singular algebraic curves, is not treated in this survey, we explain the parchworking method in several interesting situations.
Furthermore, we mention important results on the existence of curves with nodal singularities on other algebraic surfaces and in the projective space, and address several related problems.

For a comprehensive treatment of these problems  and detailed proofs, and more generally of the theory of topologically and analytically equisingular families of curves on surfaces, see the monograph \cite{GLS3}. In the present survey, we basically follow the main concept of the monograph \cite{GLS3} providing more details in certain places, for instance, in Section \ref{ssec:3.1} as well as in Section \ref{ssec:3.2}, where Theorems \ref{t-ordinary1}, \ref{t-sqh} and \ref{t-squ1} are new.\\

\subsection*{Preliminaries: isolated singularities} \label{prelim}

We work
mainly with algebraic varieties (not necessarily reduced ore irreducible) but use the Euclidean topology and analytic structure sheaf  (unless otherwise stated), called {\em algebraic complex spaces}\index{algebraic!complex space} (see \cite[Notations and Conventions]{GLS3} for a precise definition). An {\em algebraic curve}\index{algebraic!curve} resp. {\em algebraic surface}\index{algebraic!surface} means an algebraic complex space of pure dimension one resp. two.
By a {\em real algebraic variety}\index{real!algebraic variety}\index{variety!real algebraic} resp. {\em real analytic variety}\index{real!analytic variety}\index{variety!real analytic}  we mean an algebraic resp. analytic variety equipped with an anti-holomorphic involution. By a {\em hypersurface}\index{hypersurface} we mean an effective Cartier divisor in a smooth variety $\Sig$.

A {\em singularity}\index{singularity} is by definition the germ
$(X,z)$ of a complex space, may be smooth. A singularity $(X,z)$ is {\em isolated} \index{isolated singularity}\index{singularity!isolated} if $X\setminus \{z\}$ is smooth for some representative $X$. Two hypersurface singularities $(X,z)\subset(\Sigma,z)$ and
$(X',z')\subset(\Sigma,z')$ are called {\it analytically equivalent}\index{analytically!equivalent}\index{equivalent!analytically}
(resp. {\it topologically equivalent}\index{topologically!equivalent}\index{equivalent!topologically}) if there exists an analytic isomorphism (resp. a homeomorphism) of  neighborhoods of $z$ resp. $z'$ in
$\Sigma$ mapping $(X,z)$ to $(X',z')$.

The analytic equivalence can be expressed as an isomorphism
of the analytic local rings:
$\ko_{X,z}\cong{\ko}_{X',z'}.$
The topological equivalence is used in this paper only for reduced plane curve singularities where it is
completely characterized by discrete invariants (see \cite{Za1}, \cite{Wa2}, \cite{Te}, \cite{BK},  \cite{GLS5}):
Namely, two reduced plane curve singularities $(C,z)$ and $(C',z')$ are {\it topologically equivalent}\index{topologically!equivalent} \index{equivalent!topologically} iff there exists a bijection of local branches such that the Puiseux pairs of the corresponding branches coincide, as well as the pairwise intersection multiplicities of the corresponding branches; equivalently if they have embedded resolutions by
blowing up points such that the systems of multiplicities of the reduced total transforms coincide.
  The second definition is the preferred one since it generalizes to deformations over non-reduced base spaces.

Analytic resp. topological equivalence classes of isolated singular
points are called {\em (contact) analytic types}\index{analytic!type}\index{type!analytic} resp. {\em topological types}\index{topological!type}\index{type!topological} (or  {\em analytic}\index{analytic!singularity}\index{singularity!analytic} resp. {\em topological} {\em singularities})\index{topological!singularity}\index{singularity!topological}. For {\em simple} or {\em ADE} singularities (cf. \cite{GLS5}) analytic and topological types coincide and we talk simply about their type. Of particular interest are the simple singularities of type $A_1$, called {\em nodes} \index{node} , given in local analytic coordinates as $x^2+y^2=0$ and of  type $A_2$, called {\em (ordinary) cusps} \index{cusp}, given as $x^2+y^3=0.$

Important numerical invariants are the Milnor number, the delta invariant and the kappa-invariant.
Let $(X,z)\subset (\Sig,z) \cong (\CC^n,0)$
be an isolated hypersurface singularity and
$f\in \CC\{x_1, \dots, x_n\} \cong \ko_{\Sigma,z}$ a defining power series in local coordinates $x_1, \dots, x_n$. Then
$$\mu(X,z):=\dim_\CC {\CC\{x_1, \dots, x_n\}}/\langle\frac{\partial
f}{\partial x_1},\dots,\frac{\partial f}{\partial x_n}\rangle$$
\index{$mu(X,z)$@$\mu(X,z)$}
is the {\em Milnor number}\index{Milnor nunber} of $(X,z)$  and
$$\tau(X,z):=\dim_\CC {\CC\{x_1, \dots, x_n\}}/\langle f, \frac{\partial
f}{\partial x_1},\dots,\frac{\partial f}{\partial x_n}\rangle
\index{$tau(X,z)$@$\tau(X,z)$}$$
the {\em Tjurina number}\index{Tjurina number} of $(X,z)$,
which is the dimension of the base space of the semiuniversal deformation of $(X,z)$.

For  a reduced curve singularity $(C,z)$ we call
$$\delta(C,z):=\dim_\CC(\nu_*{\ko}_{\overline C}/{\ko}_C)_z$$
the {\em delta-invariant ($\delta$-invariant)}\index{delta invariant@$\delta$-invariant} \index{$d(C,z)$@$\delta(C,z)$}  \index{delta!invariant}
of
$(C,z)$, where $\nu:\overline C\to C$ is the normalization of a small representative $C$ of $(C,z)$. Let $(C,z)$ be a reduced plane curve singularity defined by $f  \in \CC\{x, y\}$. The {\em kappa-invariant ($\kappa$-invariant)}
\index{kappa invariant@$\kappa$-invariant} \index{$k(C,z)$@$\kappa(C,z)$}  \index{kappa invariant}
of $(C,z)$ is the intersection multiplicity of $f$ with a generic polar, that is,
\begin{equation}\kappa(C,z) := \dim_\CC{\CC\{x, y\}}/\langle f,\alpha\frac{\partial f}{\partial x}+ \beta\frac{\partial f}{\partial y}\rangle,\label{kappa1}\end{equation}
with $(\alpha:\beta) \in \PP^1$ generic. We also write $\mu(f), \ \del(f)$ and $\kappa(f)$. Recall for a plane curve singularity $f$ the formulas (cf. \cite{Mi} and  \cite[Proposition I. 3.35 and Proposition I. 3.38]{GLS5})
\begin{eqnarray}
\mu(f) & = 2\del(f) - r(f) +1,\nonumber\\
\kappa(f) & = \mu(f) +\mt(f) -1,\nonumber
\end{eqnarray}
where $r(f)=r(C,z)$ is the {\em number of branches} of $(C,z)$ (irreducible factors of $f$) and $\mt(f)=\mt(C,z)$ the {\em multiplicity} of  $(C,z)$ (degree of lowest non-vanishing term of $f$).

We introduce further the {\em tau-es-invariant ($\tau^{es}$-invariant)}
$$\tau^{es}(C,z) := \tau(C,z) - \dim_\CC T^{1,es}(C,z)
= \dim_\CC \ko_{\Sig,z}/I^{es}(f), $$
\index{tau-es invariant@$\tau^{es}$-invariant} \index{$tes(C,z)$@$\tau^{es}(C,z)$} \index{tau-es invariant}with $I^{es}(f)$ the equisingularity ideal (cf. \cite[Definition 1.1.63]{GLS3}) and  $T^{1,es}(C,z)$ the tangent space to the equisingular stratum (= the $\mu$-constant stratum) $\Delta^\mu$ in the base of the semiuniversal deformation of $(C, z)$. Since $\Delta^\mu$
is smooth, $\tau^{es}(C,z)$ is equal to the codimension of the $\mu$-constant stratum in the ($\tau$-dimensional) base space of the semiuniversal deformation of $(C, z)$, which coincides with the  codimension of the $\mu$-constant stratum in the ($\mu$-dimensional) base space of the semiuniversal unfolding of $f$. We have
also (cf. \cite[Lemma 1.3]{GL})
$$ \tau^{es}(C,z) = \mu(C,z) - \text{ modality}(f), $$
where  {\em modality}($f$) is the  modality of the function $f$ with respect to right equivalence. Note that $\tau^{es}(C,z)$ can be effectively computed in terms of the resolution invariants of $(C,z)$), an algorithm is implemented in {\sc Singular} \cite{DGPS}. For details we refer to \cite[ Remark to Corollary II.2.71]{GLS5} and to \cite[Corollary 1.1.64]{GLS3}.

Now we  can explain more precisely the {\em $T$-smoothness property}.\index{T-smoothness} Let $S$ be
an analytic resp. topological singularity type of a plane curve singularity $(C,z)$.
The requirement that a curve of degree $d$ has a singularity of type $S$ imposes $\tau(S):=\tau(C,z)$ resp. $\tau^{es}(S):=\tau^{es}(C,z)$  conditions on the space of all curves of degree $d$ (cf. \cite{GK} for anaytic types and  \cite{GL} for toplogical types).
Let $S_1,...,S_q$ be analytic types and
$S_{q+1},...,S_r$ topolgical types of the degree $d$-curve $C\subset \PP^2$. Then we say that
$V_d^{(irr)} (S_1,\dots,S_r)$ has the {\em expected dimension}\index{dimension!expected}\index{expected dimension} at $C$ if its dimension at $C$ is
\begin{eqnarray}
\frac {d^2+3d}{2}
- \sum_{i=1}^q \tau(S_i) - \sum_{i=q+1}^r \tau^{es}(S_i), \label{expdim}
\end{eqnarray}
and $V_d^{(irr)} (S_1,\dots,S_r)$ is {\em $T$-smooth at $C$}\index{T-smooth} if it is smooth of expected dimension at $C$
(in particular, the number (\ref{expdim}) must be non-negative). We refer to  \cite[Corollary 6.3]{GK},  \cite[Theorem 3.6]{GL}, and \cite[Theorem 2.2.40]{GLS3} for this and for further properties of $V_d^{(irr)} (S_1,\dots,S_r)$.

\section{Singular plane curves: restrictions}  \label{sec.arb:restr}\index{restrictions!for existence}

Various restrictions for the existence of plane curves of degree $d$ with prescribed
singularities $S_1,\dots,S_r$ have been found. We recall the most important ones.

\subsection{Genus formula and B{\'e}zout's theorem}
First, one should mention the general classical bound\index{genus!bound}\index{bound!genus}
\begin{equation}
\label{genbound}
\sum_{i=1}^r \delta(S_i) \: \leq \:
\frac{(d\!\!\;-\!\!\;1)(d\!\!\;-\!\!\;2)}{2}\,,
\end{equation}
for the existence of an irreducible plane curve of degree d havings $r$ singularities of type $S_1, \dots, S_r$, which
results from the genus formula (\ref{genusformula}).

For a  reduced (not necessarily irreducible) plane curve we get
as necessary bound for the existence, i.e., for $V_d(S_1,\dots S_r) \neq \emptyset$,
the inequality\index{Bezout!bound}\index{bound!Bezout}
\begin{equation}
\label{mubound}
\sum_{i=1}^r \mu(S_i)  \: \leq \: (d\!\!\;-\!\!\;1)^2.
\end{equation}
This is a consequence of {\em B{\'e}zout's theorem} (see e.g. \cite[Theorem II. 1.16]{GLS3})\index{Bezout!theorem}:
\medskip

{\em Two plane projective curves $C,D \subset \PP^2$ of degrees $c$ and $d$, respectively, which have no component in common, intersect at $c\cdot d$
points, counting intersection multiplicities. That is,
 \begin{equation}
\label{bezout}
c\cdot d=\sum_{z\in C\cap D}\dim_\CC \,\ko_{\PP^n,z}/\langle f,g \rangle,
\end{equation}
with $f$ resp. $g$ being local equations of $C$ resp. $D$ at $z$.}
\medskip

To see (\ref{mubound}) let $C$ be given by a homogeneous polynomial $F \in \CC[x_0,x_1,x_2]$ of degree $d$ and let
$F'_\alpha=\sum_{i=0}^2\alpha_i \partial F/\partial x_i $ and
$F'_\beta = \sum_{i=0}^2\beta_i \partial F/\partial x_i$ with
$\alpha_i , \beta_i$ generic, be two generic
polars of $C$, both of degree $d-1$. The intersection points of $\{F'_\alpha = 0\}$ and $\{F'_\beta = 0\}$ include the singular points of $C$ and the intersection multiplicities are just the corresponding Milnor numbers. Thus, we get
(\ref{mubound}).

\bigskip

For the proof of (\ref{genbound}) let us deduce the genus formula. Recall that for an arbitrary projective scheme $X$ the {\em arithmetic genus} \index{arithmetic genus} \index{genus!arithmetic} is defined as
$$p_a(X) := (-1)^{\dim X} (\chi(\ko_X)-1).$$
Here, for any coherent sheaf $\kf$ on $X$,
$\chi(\kf) = \sum (-1)^i \dim_\CC H^i(X, \kf)$
 is the {\em (algebraic) Euler characteristic} \index{Euler characteristic!analytic} of $\kf$.
 \medskip

 For a curve $C$ we have $p_a(C) = 1- \chi(\ko_C) = 1- \dim_{\CC} H^0(C,\ko_C) + \dim_{\CC} H^1(C,\ko_C)$. If $C$ is reduced and connected, then we have  $H^0(C,\ko_C) =\CC$ and  hence we get for the arithmetic genus
$p_a(C) = \dim_{\CC} H^1(C,\ko_C) \geq 0.$  If $C$ has $s$ connected components $C_1,...,C_s$,  the additivity of the Euler characteristic implies
$p_a(C) = 1-s  + \dim_{\CC} H^1(C,\ko_C) =
1-s + \sum_{i=1}^sp_a(C_i)$, which may be negative for $s>1$.
\medskip

 The {\em geometric genus} \index{genus!geometric} of a reduced curve $C$ is defined as the arithmetic genus of the normalization $\overline C$ of $C$, hence \index{genus!formula}
 \begin{equation}
\label{geomgenus}
 g(C) := p_a(\overline C)  = p_a(C) - \delta(C),
\end{equation}
with $\delta(C) := \dim_\CC H^0(\nu_*\ko_{\overline C}/\ko_C)$ the {\em total delta invariant of C}\index{delta!invariant!total} and
$\nu: \overline C \to C$ the normalization map.
(\ref{geomgenus}) follows from applying $ \chi$ to the exact sequence
$$0\to \ko_C \to \nu_*\ko_{\overline C} \to \nu_*\ko_{\overline C}/\ko_C \to 0,$$
noting that $ \chi (\nu_*\ko_{\overline C} ) = \chi(\ko_{\overline C})$ and
$H^i(\nu_*\ko_{\overline C}/\ko_C) = 0$ for $i>0$. Moreover,
$\delta(C) = \sum_{z\in\Sing(C)}\!\del(C,z),$
with $\del(C,z)$ the local  delta invariant of $C$ at $z$.\index{delta!invariant!local}
For a smooth curve $C$  the arithmetic genus and the geometric genus coincide ($\delta(C) =0$).

\medskip
If $C$ is irreducible, then $\overline C$ is  connected and smooth and $g(C) = p_a(\overline C) = g(\overline C)\geq 0$.
If $C$ is a reduced curve with $s$ irreducuible components $C_1,...,C_s$, we have  \index{genus!formula}
\begin{equation}
\label{geomgenus_red}
g(C)  = 1-s + \sum_{i=1}^sg(C_i)
\end{equation}
and hence $g(C) +s-1 \geq 0.$ The general genus formulas (\ref{geomgenus}) and (\ref{geomgenus_red}) were first proved by Hironaka \cite[Theorem 2]{Hir} using the resolution of the singularities of $C$ (he defines $g(C)$ as $\sum_{i=1}^sg(C_i)$).

\medskip
 If $C \subset \PP^2$ is a  {\em plane curve} of degree $d >0$, then $C$ is connected (by Bezout's theorem) and we have
 \begin{equation}
\label{pa} \index{genus!arithmetic}
p_a(C) =\: \frac{(d\!\!\;-\!\!\;1)(d\!\!\;-\!\!\;2)}{2}.
\end{equation}
 This follows from the exact sequence
 $$0 \to \ko_{\PP^2}(-d) \to \ko_{\PP^2} \to \ko_C \to 0,$$
giving $1-p_a(C) = \chi(\ko_C) = \chi(\ko_{\PP^2}) - \chi(\ko_{\PP^2}(-d))
= 1- \chi(\ko_{\PP^2}(-d))$, and from
$\chi(\ko_{\PP^2}(-d)) = \dim_{\CC} H^2(\PP^2,\ko_{\PP^2}(-d))$ =
${d-1} \choose {2}$.
\medskip

Now, if $C \subset \PP^2$ is reduced and irreducible, then $\overline C$ is smooth and connected and the geometric genus $g(C) = g(\overline C)$ is non-negative. The formulas (\ref{geomgenus}) and  (\ref{pa}) imply the
{\em genus formula}\index{genus!formula}
\begin{equation}
\label{genusformula}
 g(C) =
\frac{(d\!\!\;-\!\!\;1)(d\!\!\;-\!\!\;2)}{2}\,- \delta(C).
\end{equation}
Since $g(C)$ is non-negative for an irreducible curve of degree $d$ we get the inequality (\ref{genbound}).
\medskip

Of course, B{\'e}zout's\index{Bezout!theorem} theorem leads to various further necessary
conditions for the existence of the curve $C$ such as, for instance
by considering a line through $2$ points or a conic through $5$ points,
$$ \max_{i\neq j} \,\bigl(\mt(S_i)+ \mt(S_j)\bigr) \:\leq\: d\,, \quad
 \max_{\#(I)=5} \:\sum_{i\in I} \mt(S_i) \: \leq \: 2d\,.
$$
Finally we mention the inequality
\begin{eqnarray}\index{bound!regular existence}\index{regular!existence}
 \sum_{i=1}^q \tau(S_i) + \sum_{i=q+1}^r \tau^{es}(S_i) \leq  \frac {d^2+3d}{2}
\label{regex}
\end{eqnarray}
for {\em regular existence}, that is, for the existence of a curve $C \subset \PP^2$ of degree $d$ with $q$ analytic singularities $S_1,...,S_q$ and $r-q$ topological singularities $S_{q+1},...,S_r$, such that
$V_d(S_1,\dots S_r)$ is $T$-smooth at $C$ (cf. (\ref{expdim}) and \cite[Corollary 6.3 (ii)]{GK}, \cite[Corollary 3.9]{GL}, \cite[Theorem 2.2.40]{GLS3}).

\subsection{Pl\"ucker formula}\index{Pl\"ucker!formula}

Besides the genus formula and B{\'e}zout's theorem the Pl\"ucker formula provide necessary bounds for the existence, which are often sharper. Let's deduce these formulas.

Let \mbox{$C\subset\PP^2$} be a reduced, irreducible curve of degree
\mbox{$d>1$}, given by a homogeneous polynomial \mbox{$F\in
  \CC[x_0,x_1,x_2]$}.
Denote by \mbox{$C^*\subset (\PP^2)^*$} its {\em dual
    curve}\index{dual!curve}, that is,
$$ C^*:=\left\{(a_0:a_1:a_2)\in (\PP^2)^* \:\left|\:
    \begin{array}{l}
V(a_0x_0+a_1x_1+a_2x_2) \text{ is tangent}\\
 \text{to $C$ at some point }p\in C
    \end{array}
\right.\right\}
$$
Here \mbox{$(\PP^2)^*$} is the {\em (dual) projective
  2-space}\index{dual!projective space},  whose points
\mbox{$(a_0:a_1:a_2)$} are in 1-1 correspondence with the lines
\mbox{$V(a_0x_0+a_1x_1+a_2x_2)\subset \PP^2$}.

We have a natural rational duality morphism
$\bd:C\dashrightarrow C^*$, mapping a (smooth) point $z$ of $C$ to its tangent at $z$:
Let $z\in C$ and $P$ an irreducible component of the germ $(C,z)$. In
local affine coordinates $x,y$ such that $z=(0,0)$ and the $x$-axis is tangent to $P$,
this component admits a parametrization
$$\begin{cases}&x=t^p\\ &y=\lambda t^q+O(t^{q+1})\end{cases}\quad 1\le p<q,\quad \lambda\ne0,\quad t\in(\CC,0)\ .$$
and the tangent lines to the points of $P$ are given by equations $y=b(t)x+c(t)$ with
\begin{equation}b(t)=\frac{\lambda q}{p}t^{q-p}+O(t^{q-p+1}),\quad
c(t)=\lambda t^q+O(t^{q+1}),\quad t\in(C,0)\ .\label{en-dual}\end{equation}
It follows that generically the duality morphism $\bd$ is 1-1, and hence birational.
Furthermore, $C^*$ is an irreducible projective curve\footnote{An equation $F^*$ for $C^*$ can be obtained
  as follows: let
$$ g(x_1,x_2) := a_0^d\cdot F\left(\frac{-(a_1x_1+a_2x_2)}{a_0},
  x_1,x_2,\right) \in \CC(a_0,a_1,a_2)[x_1,x_2]\,,$$
and compute the discriminant \mbox{$D\in
\CC[a_0,a_1,a_2]\setminus\{0\}$} of $g(1,x_2)$. $D$ is homogenous of
degree \mbox{$2d^2\!-d$} and is the product of $F^*$ with some
number of linear factors. Hence, factorizing $D$ and removing all
linear factors, we get an equation for $C^*$.} of degree
\mbox{$d^*\!>1$}. The degree $d^*$ of $C^*$
is classically called the {\em class of $C$}.\index{class!of a curve}

Let $C=V(F)$ with $F(x_0,x_1,x_2)$ a homogeneous polynomial of degree $d$.
At a smooth point $z\in C$, the coefficients of the tangent line are given by
$$a_0=\frac{\partial F}{\partial x_0},\quad a_1=\frac{\partial F}{\partial x_1},\quad a_2=\frac{\partial F}{\partial x_2}\ .$$
Let $(\CC,0) \ni t \mapsto z(t)= (x_o(t),x_1(t),x_2(t))$ parametrize the germ $(C,z)$. Since $F(z(t))\equiv0$, we have
$\frac{\partial F(z(t))}{dt} = \frac{\partial F}{\partial x_0}\cdot \frac{dx_0}{dt}+
\frac{\partial F}{\partial x_1}\cdot \frac{dx_1}{dt}+
\frac{\partial F}{\partial x_2}\cdot \frac{dx_2}{dt}=0$, {that is, with $a_i(t) = a_i(z(t))$,
\begin{equation}\label{par}
a_0(t)\cdot\frac{dx_0}{dt}+
a_1(t)\cdot\frac{dx_1}{dt}+a_2(t)\cdot\frac{dx_2}{dt}=0\ .
\end{equation}
Combining this with the Euler formula
$dF=x_0\partial F/\partial x_0 + x_1\partial F/\partial x_1 + x_2\partial F/\partial x_2$, hence
$a_0(t)x_0(t)+a_1(t)x_1(t)+a_2(t)x_2(t)=0$, we obtain
\begin{equation}\label{par-dual}
x_0(t)\cdot\frac{da_0}{dt}+x_1(t)\cdot\frac{da_1}{dt}+x_2(t)\cdot\frac{da_2}{dt}=0\ ,
\end{equation}
which is dual to (\ref{par}).
Thus, the dual to $C^*$ is the original curve $C$.

We call a tangent line $L$ to $C$ a {\em singular
  tangent\/}\index{singular!tangent}\index{tangent!singular}, if
\begin{quote}
\begin{enumerate}
\item[(a)]  either $L$ is tangent to $C$ at a singular point,
\item[(b)] or $L$ is tangent to $C$ at more than one point,
\item[(c)]  or $L$ intersects $C$ at a non-singular point with multiplicity
  $>2$.
\end{enumerate}
\end{quote}
The set of singular
tangents is finite, since the set $\Sing(C)$ is finite, and the conditions
(b) and (c) determine $L$ as a singular point of $C^*$ (cf. formula (\ref{en-dual})). Hence, there exists a point
\mbox{$\bq=(q_0\!:\!\!\:q_1\!:\!\!\:q_2)\in\PP^2\setminus C$} which
does not lie on any singular tangent. Denote by $\Lam_{\bq}$ the
pencil of lines through $q$. Recall that a line
\mbox{$L\in\Lam_{\bq}$} is tangent to $C$ at the (non-singular) point
\mbox{$z\in C$} iff $z$ lies also on the polar\index{polar!curve} curve
relative to $\bq$, that is, iff
$$F(z)\,=\,0 \,=\, q_0\frac{\partial F}{\partial x_0}(z)+
q_1\frac{\partial F}{\partial x_1}(z)+ q_2\frac{\partial F}{\partial
x_2}(z)\, .$$ We observe that \mbox{$d^*$} is the number of lines
\mbox{$L\in\Lam_{\bq}$} tangent to $C$ at non-singular points, and
that \mbox{$\{q_0\frac{\partial F}{\partial x_0}\!+q_1\frac{\partial
F}{\partial x_1} \!+q_2\frac{\partial F}{\partial x_2}=0\}$} is a
generic polar of $C$. Applying B\'ezout's  theorem
(\ref{bezout}) to the non-singular intersection points of $\{F=0\}$ with
a generic polar (which gives $d^*$ as total number) and the singular
intersection points (which gives the kappa-invariant $\kappa$ at each intersection point, we obtain the {\em first Pl\"ucker
formula}\index{Pl\"ucker!formula} (cf. (\ref{kappa1})):
\begin{equation}
d^*=d(d-1)-\sum_{z\in\Sing(C)}\kappa(C,z).\label{plucker}
\end{equation}
In particular, if $C$ has $n$ nodes and $k$ cusps as its only
singularities one gets
\begin{equation}
d^*=d(d-1)-2n-3k. \label{e304}
\end{equation}
\medskip

We derive now the Riemann-Hurwitz formula and give another proof of the genus formula.
Let \mbox{$\overline{C}\to C$} be the normalization
map. Then the {\em topological Euler characteristic}\index{Euler characteristic!topological}\index{topological!Euler characteristic} of  $C$ satisfies (using Mayer-Vietoris)
\begin{eqnarray}
\chi_{\text{top}}(C)&=& \chi_{\text{top}}\bigl(\overline{C}\bigr)-
\!\!\sum_{z\in\Sing(C)}\!\bigl(r(C,z)\!\!\:-\!\!\:1\bigr)
\nonumber\\
&=&
2-2g(C)-\!\!\sum_{z\in\Sing(C)}\!\bigl(r(C,z)\!\!\:-\!\!\:1\bigr)\,
,\label{e305}
\end{eqnarray}
where $r(C,z)$ is the number of irreducible branches of the germ
$(C,z)$. Besides, considering the projection of $C$ on some
straight line \mbox{$L_0\not\supset\{\bq\}$} from the point $\bq$
leads to the following version of the topological
{\em Riemann-Hurwitz formula}\index{Riemann-Hurwitz formula},
$$\chi_{\text{top}}(C)\,=\,d\cdot \chi_{\text{top}}(L)-d^*-
\!\!\sum_{z\in\Sing(C)}\!\bigl(\mt(C,z)\!\!\:-\!\!\:1\bigr)\, ,$$
since a line \mbox{$L\in\Lam_{\bq}$}, which is tangent to $C$ at a
non-singular point, meets $C$ at \mbox{$d\!\!\:-\!\!\:1$} points, and
a line \mbox{$L\in\Lam_{\bq}$} through a point
\mbox{$z\in\Sing(C)$} meets $C$ at
\mbox{$d\!\!\:-\!\!\:\mt(C,z)\!\!\:+\!\!\:1$} points.
Combining the last equation with (\ref{plucker}) and (\ref{e305}),
we come to the  genus formula (\ref{genusformula}), \index{genus!formula}
$$g(C)=\frac{(d-1)(d-2)}{2}-\!\!\sum_{z\in\Sing(C)}\!\del(C,z).$$
\medskip

We also mention the \emph{second Pl\"ucker formula}\index{Pl\"ucker!formula}: for any reduced plane curve of degree $d\ge3$ which does not contain lines as components, the following equality holds
\begin{equation}\sum_{z\in\Sing(C)}h(C,z)=2d(d-2)-\sum_{z\in C\setminus\Sing(C)}((C\cdot T_zC)_z-2),
\label{eq:hess}\end{equation}
where $(C\cdot T_zC)_z$ stands for the intersection number of the curve $C$ with its tangent line $T_zC$ at the point $z$. According to \cite{Sh14}, if $z\in\Sing(C)$, then
\begin{equation}h(C,z)=3\kappa(C,z)+\sum_{C'}((C'\cdot T_zC')_z-2\mt(C',z)),\label{eq:hesssh}\end{equation}
where $C'$ ranges over all local branches of $C$ at $z$ (i.e., irreducible components of the germ $(C,z)$).

 The second Pl\"ucker equation for a reduced plane curve of degree $d\ge3$ which does not contain lines as components and with $n$ nodes and $k$ cusps states\index{Pl\"ucker!formula}
\begin{equation}
    k^* = 3 d (d - 2 ) - 6n - 8k,
    \label{kdual}
\end{equation}
where $k^*$ is the number of cusps of $C^*$. This follows from (\ref{eq:hess}) and (\ref{eq:hesssh}): indeed, when there are no flexes at
the nodes, and at all smooth flexes we have a triple intersection with the tangent, then $h(A_1)=3\kappa(A_1)+0=6$, $h(A_2)=3\kappa(A_2)-1=8$.

\subsection{Miyaoka-Yau inequality}
By applying the log-Miyaoka inequality, F.\ Sakai \cite[Theorem A]{Sa}
obtained the necessary condition
\index{bound!Miyaoka-Yau}
$$\sum_{i=1}^r \mu(S_i) \: < \: \frac{2\nu}{2\nu+1} \cdot \left(d^2\! -\frac{3}{2}\!\: d \right)
$$
where $\nu$ denotes the maximum of the multiplicities $\mt (S_i)$, \mbox{$i=1,\dots,r$}. In \cite{Sa} further bounds for the total Milnor number are given. In particular, if \mbox{$S_1, \dots,S_r$} are $ADE$-singularities then
$$\sum_{i=1}^r \mu(S_i) \: < \: \left\{
\renewcommand{\arraystretch}{2.0}
\begin{array}{cl}
\displaystyle{
\frac{3}{4}\!\:d^2\! -\frac{3}{2} \!\: d +2} & \text{ if $d$ is
even}\,,\\
\displaystyle{\frac{3}{4}\!\:d^2\! - d +\frac{1}{4}} & \text{ if $d$
  is odd}\,,
\end{array}
\right.
$$
is necessary for the existence of a plane curve with $r$ singularities of types
\mbox{$S_1, \dots,S_r$}.

Applying the strengthened Bogomolov-Miyaoka-Yau inequality in the form of Miyaoka \cite{Miy} to the desingularized double covering of the plane ramified along a curve with simple singularities, i.e., $A_r$, $D_r$, $E_6$, $E_7$, $E_8$, Hirzebruch and Ivinskis \cite{HF,Iv} obtained the following bound for a reduced plane curve $C$ of an even degree $d\ge6$ having only simple singularities:
\begin{equation}\index{bound!Hirzebruch-Ivinskis}
\sum_{z\in\Sing(C)}m(C,z)\le\frac{d(5d-6)}{2}\ ,
\label{eq:HI}\end{equation}
where the invariant $m(C,z)$ can be computed as follows:
\begin{equation}m(A_r)=\frac{3r(r+2)}{r+1},\quad m(D_r)=\frac{3(4r^2-4r-9)}{4(r-2)},\label{eq:Iv1}\end{equation}
$$m(E_6)=\frac{167}{8},\quad m(E_7)=\frac{383}{16},\quad m(E_8)=\frac{1079}{40}.$$

Langer \cite[Theorem 1]{Lan} generalized the Bogomolov-Miyaoka-Yau inequality to orbifold
Euler numbers and obtained an upper bound to the number of simple singularities of curves on surfaces.
In particular (see \cite[Theorem 9.4.2
and formula (11.1.1)]{Lan}), for any reduced curves of degree $d\ge10$ with $n$ nodes and $k$ cusps it yields
the bound
\begin{equation}
(2-\alpha)\,n+\left(\frac{7}{2}-\frac{3}{2}\alpha-\frac{1}{24\alpha}\right)k\le\left(1-\frac{\alpha}{3}\right)d^2-d
\label{elan1}
\end{equation}
with an arbitrary $\frac{3}{10}\le\alpha\le\frac{5}{6}$, which is always better than Hirzebruch-Ivinskis' bound
(\ref{eq:HI}).
Substituting $\alpha=\frac{\sqrt{73}-1}{24}$, one obtains the maximal coefficient of $k$ in (\ref{elan1}), and hence
\begin{equation}\index{bound!Langer}
\frac{6059+7\sqrt{73}}{10512}\,n+k\le\frac{125+\sqrt{73}}{432}\,d^2-\frac{511+11\sqrt{73}}{1752}\,d\ .
\label{elan2}\end{equation}

\subsection{Spectral bound} \index{bound!spectral}
Further necessary conditions can be obtained by applying the semicontinuity of the singularity spectrum (see \cite{Va}), which works in any dimension. The  singularity spectrum of a hypersurface singularity $f:\CC^{n+1} \to \CC$  gathers the information about the eigenvalues of the monodromy operator $T$ and about the Hodge filtration
$\{F^p\}$ on its vanishing cohomology. The  {\em (singularity) spectrum} \index{singularity!spectrum} \index{spectrum} is defined as an unordered $\mu(f)$-tuple of rational numbers $(a_1,...,a_\mu)$  (counted with frequencies), where
the frequency of the number $a$ in the spectrum is equal to the dimension of the eigenspace of the semisimple part of $T$ acting on $F^p/F^{p+1}, \, p=[n-a]$, with eigenvalue $exp(-2\pi i a)$.
If $F: X \to S$  is a good representative of a deformation of $f$,  let $\Sig_{F^{-1}(s)}$ denote the union of all spectra of the singular points in the fiber $F^{-1}(s)$, where the frequency of $a$ in
$\Sig_{F^{-1}(s)}$ is the sum of its frequencies in the spectra of all singular points of $F^{-1}(s)$.

The {\em semicontinuity of the spectrum} \index{spectrum!semicontinuity}  says that any half open interval $(t,t+1] \subset \R$ is a semicontinuity domain for $F$, that is, the sum $M_{F^{-1}(s)}$ of the frequencies of the elements of $(t,t+1]$ in $\Sig_{F^{-1}(s)}$  is upper semicontinuous for $s\in S$ (\cite[Theorem 2.4]{St}).
Before that Varchenko \cite{Va} had proved that for deformations $F$ of low weight of a quasi-homogeneous $f$ even every open interval $(t,t+ 1)$ is a semicontinuity domain.\\

The semicontinuity of the singularity spectrum can be used to
compute effectively an upper bound for the number of isolated hypersurface
singularities of a given type occurring on a hypersurface
\mbox{$V\subset \PP^n$} of degree $d$. More precisely, one can apply
the following two facts:
\begin{itemize}
\item any hypersurface of degree $d$ with isolated singularities can be obtained
  as small deformation of $V(f_d)$, where $f_d(x_1,...,x_n)$ is
  a nondegenerate $d$-form, which stays fixed in the deformation, while all variable terms have degree $<d$\,\footnote{Such deformations are called \emph{lower}.}; furthermore, the space of nondegenerate $d$-forms is connected;
  hence for the bound it is sufficient to compute the spectrum of
  $f=x_1^d+\ldots+x_n^d$;
\item if precisely $M$ of the spectral
  numbers (counted with their frequencies) of the singularity defined by
  \mbox{$f=x_1^d+\ldots+x_n^d$}
  are in the interval \mbox{$(t,t+1]$}, \mbox{$t\in \R$},  then the sum of the frequencies of the spectral numbers in the interval $(t,t+1]$ of the  singularities (close to $0$)
  of a small deformation $F$ of $f$ can be at most $M$, i.e.,
  $M_{F^{-1}(s)} \leq M_{f^{-1}(0)}=M$
(cf. \cite[Theorem 2.4]{St}; if the deformation is of low weight, we can use even the open interval $(t,t+1)$ by \cite{Va}).
\end{itemize}

For instance, we can look for an upper bound for the number of cusps which may
appear on a curve of degree $11$ by using \textsc{Singular} \cite{DGPS,Schu}\index{Singular@\textsc{Singular} code!for spectrum}:

\begin{small}\begin{verbatim}
LIB "gmssing.lib";
ring r=0,(x,y),ds;
poly g=x^2-y^3;         // a cusp
list s1=spectrum(g);    // spectral numbers of a cusp (with mult's)
s1;
//-> [1]:
//->      _[1]=-1/6  _[2]=1/6
//-> [2]:
//->      1,1
\end{verbatim} \end{small}
That is, for a cusp we have $M_{g^{-1}(0)} =2$ for each interval $(t,t+1)$ with $t< -1/6$, $t+1 > 1/6$.
\begin{small}\begin{verbatim}
poly f = x^11+y^11;
list s2 = spectrum(f);  // spectral numbers of f (with mult's)
s2;
//-> [1]: (spectral numbers)
//->     _[1]=-9/11  _[2]=-8/11  _[3]=-7/11  _[4]=-6/11  _[5]=-5/11
//->     _[6]=-4/11  _[7]=-3/11  _[8]=-2/11  _[9]=-1/11  _[10]=0
//->     _[11]=1/11  _[12]=2/11  _[13]=3/11  _[14]=4/11  _[15]=5/11
//->     _[16]=6/11  _[17]=7/11  _[18]=8/11  _[19]=9/11
//-> [2]: (frequencies or multiplicities)
//->     1,2,3,4,5,6,7,8,9,10,9,8,7,6,5,4,3,2,1
\end{verbatim}
  \end{small}
  Having computed the spectral numbers, we look for an
appropriate interval \mbox{$(t,t+1)$} to apply the semicontinuity theorem $M_{F^{-1}(s)} \leq M_{f^{-1}(0)}$, $F$ a deformation of $f$. If $F^{-1}(s)$ contains $k$ cusps,
then $M_{g^{-1}(0)} k =2k \leq M_{f^{-1}(0)}$.
Choosing \mbox{$t=-\frac{2}{11}$} we get $2k \leq 63$, i.e.,  at most $31$ cusps can appear on a curve of degree 11.

The same result can be computed by using
the {\sc Singular} procedure
\texttt{spsemicont} to get the sharpest bound for the number of singularities obtainable in the above way:
\begin{small}
\begin{verbatim}
spsemicont(s2,list(s1),1);
// -> [1]: 31
\end{verbatim}
\end{small}

We recall that the spectrum is a topological invariant of the curve singularity, and, for example, according to \cite{St0} for the quasihomogeneous singularity $x^m+y^n=0$ is the multiset (set with frequencies)
\begin{equation} \index{spectrum}
\left\{\frac{i}{m}+\frac{j}{n}-1\ :\ 1\le i\le m-1,\ 1\le j\le n-1\right\}.\label{eq:spectrum}\end{equation}
A simple algorithm for computing the spectrum of an arbitrary isolated curve singularity was suggested in \cite{Ker}.

Varchenko \cite{Va} used the semicontinuity of the spectrum to give an upper bound for the number of  nondegenerate singular points (i.e. of type $A_1$) on arbitrary hypersurfaces in $\PP^n$.
Let $N_n(d)$ be the maximal number of singular points of type $A_1$ which can exist on a hypersurface in $\PP^n$ of degree $d$. He proves the following  inequality (conjectured by Arnold),
\begin{equation} \index{bound!spectral}
N_n(d) \leq A_n(d) =a_nd^{\,n}+\text{ (lower degrees of {\em d})},
\label{eq:var1}
\end{equation}
with  $a_n\sim \sqrt {(6/\pi)n}=1.3819...\sqrt n, \text{ if } n \to \infty$. 

\section{Plane curves with nodes and cusps}\label{sec:2}

The simplest singularities, the node $A_1$ and the ordinary cusp $A_2$, typically occur in most of the questions related to singular curves. The case of curves with nodes and cusps is also the most studied case, both in classical and in modern algebraic geometry. It suggests beautiful results and challenging problems.
Furthermore, the study of the particular case of curves with nodes and cusps led to the development of important techniques and the discovery of most interesting phenomena in the geometry of singular algebraic curves and their families. We shall demonstrate this for the problem of the existence of a plane curve of a given degree with a given collection of nodes and cusps, both in the complex and real case.

\subsection{Plane curves with nodes}\label{ssec:2.1}

We start with the construction of complex plane curves with only nodes as singularities and with any prescribed number being allowed by the genus bound. The construction is due to Severi
\cite{Se} and very simple. It uses, however, the $T$-smoothness of a family of nodal curves in an essential way, called classically the ''completeness of the characteristic linear series''. For a modern proof see \cite{GK} and \cite[Section 4.5.2.1]{GLS3}.

For real curves their existence with the number of nodes below or equal to the genus bound is also classically known and due to Brusotti \cite{Br}, using the $T$-smoothness of the family of real nodal curves.
But in the real case we have to distinguish between three kinds of nodes: hyperbolic, elliptic and non-real (coming in complex conjugate pairs). The fact that, subject to the genus bound,  any prescribed distribution among the three different kinds can be realized was proved much later in
\cite{Sh6}, and with a different method by Pecker \cite{Pe1,Pe3,Pe4}.
The construction is much more difficult than in the complex case. It uses a ``patchworking construction'' invented by Viro for non-singlar real curves and  extended to singular curves in \cite{Sh6, Sh12}
(see also \cite[Sections 2.3 and 4.5.1]{GLS3}).\\

\noindent
{\bf Complex curves}
\medskip

Let $C$ be a complex plane irreducible curve of degree $d$ with $n$ nodes. The genus bound (\ref{genbound}) yields
\begin{equation}n\le\frac{(d-1)(d-2)}{2},\label{nodalbound}\end{equation}
since $g(C)\ge0$ and the $\delta$-invariant of the node is $1$. If we assume that $C$ consists of $s$ irreducible components, then we get \index{bound!genus}
\begin{equation}n\le\frac{(d-1)(d-2)}{2}+s-1.\label{nodalbound-reducible}\end{equation}
It goes back to Severi \cite{Se} that the bounds (\ref{nodalbound}) and (\ref{nodalbound-reducible}) are, in fact, necessary and sufficient for the existence of a plane curve of degree $d$ with $n$ nodes. More precisely,

\begin{theorem}\label{t-Severi}\index{existence!of nodes}
The bound (\ref{nodalbound}) is necessary and sufficient for the existence of a complex plane irreducible curve of degree $d$ with $n$ nodes as its only singularities.

Furthermore, for any $s\ge2$, any positive integers $d,d_1,...,d_s$ satisfying $d=d_1+...+d_s$, and nonnegative integers $n_1,...,n_s$, the inequalities
$$n_i\le\frac{(d_i-1)(d_i-2)}{2},\quad i=1,...,s,$$ are necessary and sufficient for the existence of a complex plane reduced curve of degree $d$ splitting into $s$ irreducible components of degrees $d_1,...,d_s$ and having
$$n=\sum_{i=1}^sn_i+\sum_{1\le i<j\le s}d_id_j$$ nodes as its only singularities, while the $i$-th component has precisely $n_i$ nodes, $i=1,...,s$.
\end{theorem}

\begin{proof}
Severi \index{construction!Severi's} proved that, given a nodal plane curve $C$ of degree $d$, the germ of the family of curves of degree $d$ having a node in a neighborhood of an arbitrary singular point of $C$, is a smooth hypersurface germ in ${\mathcal O}_{\PP^2}(d)|\simeq\PP^{d(d+3)/2}$, and, moreover, all these germs intersect transversally at $C$ (for a modern treatment see \cite[Corollary 6.3]{GK} and \cite[Corollary 4.3.6]{GLS3}). This fact immediately yields that there exists a deformation of $C$ in $\PP^{d(d+3)/2}$ along which prescribed nodes are smoothed out, while the others persist (possibly changing their position). Thus, given $n$ and $d$ satisfying (\ref{nodalbound}), we take the union of $d$ straight lines in general position, which is a curve with $\frac{d(d-1)}{2}$ nodes (the maximum by (\ref{nodalbound-reducible})). Then choose some line and deform the curve by smoothing out all $d-1$ intersection points of this line with the other lines, obtaining an irreducible, rational curve with $\frac{(d-1)(d-2)}{2}$ nodes (see Figure \ref{fig:lines}). Finally, we take another deformation by smoothing out $\frac{(d-1)(d-2)}{2}-n$ nodes and obtain an irreducible curve of degree $d$ with $n$ nodes as desired.

In the reducible case, we take irreducible curves of degrees $d_1,...,d_s$ with $n_1,...,n_s$ nodes respectively and place them in general position in the plane.
\end{proof}

\begin{figure}[ht] 
\bigskip
\begin{center}
\includegraphics[height=4.5cm,width=12cm]{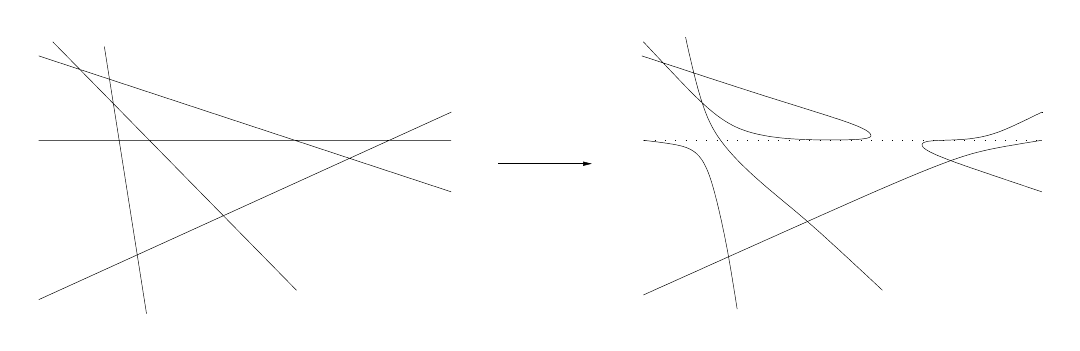}
\caption {Severi's construction of irreducible nodal curves \label{fig:lines}}
\end{center}
\end{figure}
\begin{remark}\label{r-tsmoothness} {\em We would like to underline the importance of the fact that each curve appeared in the Severi's construction was a member of a smooth equisingular family in $\PP^{d(d+3)/2}$ of expected dimension, and the germ of the family was the transversal intersection of smooth equisingular families corresponding to individual singular points of the given curve. That is, given such a curve with a possibly maximal number of singularities, one immediately obtains the existence of curves with any smaller amount of singularities.  We shall see later how efficient this property (which follows from \emph{$T$-smoothness})\index{T-smoothness} is for the construction of curves with arbitrary singularities.}
\end{remark}
\begin{figure}[ht]
\begin{center}
\unitlength1cm
\begin{picture}(11,5.5)
\put(0.7,0.5){\includegraphics[trim=20 20 20 20,
  clip,height=3cm,width=3.7cm]{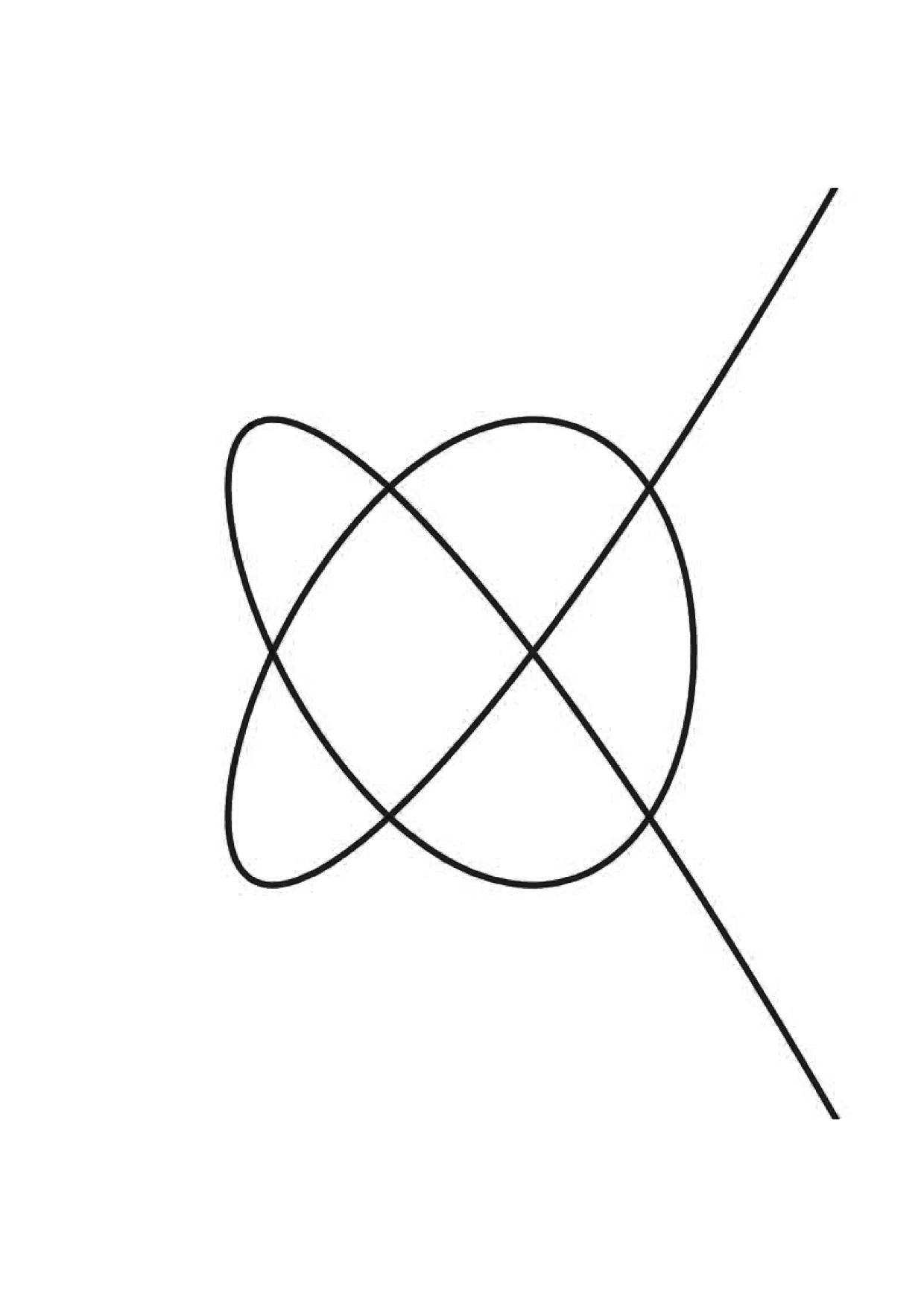}}
\put(2.5,0.2){(a)}
\put(6,0.5){\includegraphics[height=3.6cm,width=4.3cm]{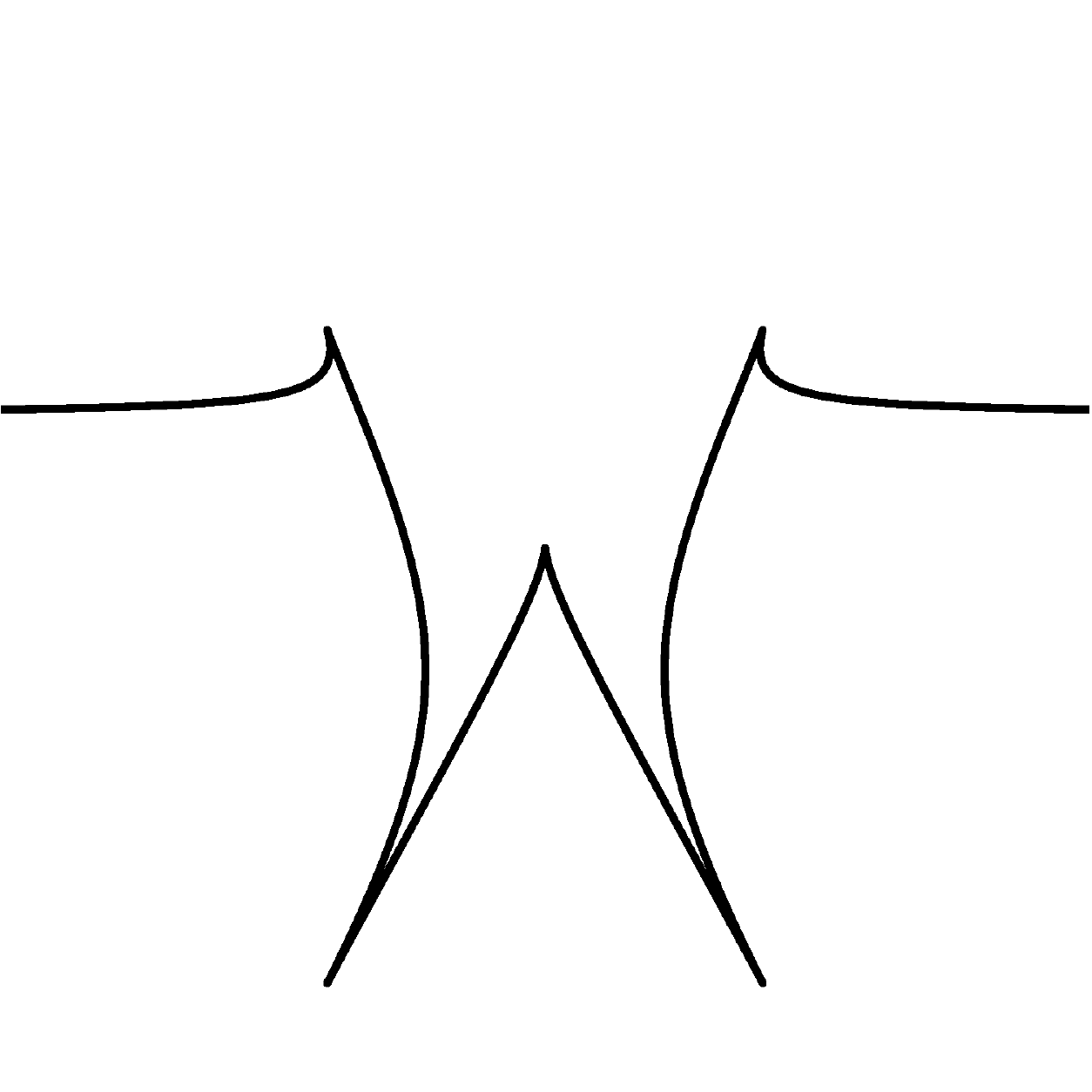}}
\put(8,0.2){(b)}
\end{picture}
\caption{Real plane quintics (a) with $6$~nodes:
$x^5-\frac{5}{4}x^3+\frac{5}{16}x-\frac{1}{2}y^4+\frac{1}{2}y^2-\frac{1}{16}=0$\,,
\ (b) with $5$ cusps: $\frac{129}{8} x^4y-\frac{85}{8} x^2y^3+
\frac{57}{32}y^5-20x^4-
\frac{21}{4}x^2y^2+\frac{33}{8}y^4-12x^2y+\frac{73}{8}y^3+32x^2=0$.}
\label{fig-quintic}
\end{center}
\end{figure}

\bigskip\noindent
{\bf Real curves}
\medskip

Over the reals, a nodal singular point can be of one of the following three types:
\begin{itemize}\item either {\em hyperbolic},\index{node!hyperbolic} i.e., a real intersection point of two smooth real local branches, locally equivalent to $\{x^2 - y^2 = 0\}$,
\item or {\em elliptic},\index{node!elliptic} i.e., a real intersection point of two complex conjugate smooth local branches, locally equivalent to $\{x^2+ y^2 = 0\}$,
\item or {\em imaginary}, \index{node!imaginary}i.e., a non-real node (which always comes in pairs of complex conjugate nodes).
\end{itemize}

Thus, a natural question is what amount of hyperbolic, elliptic and pairs of complex conjugate nodes a real plane curve can have. Note that Severi's construction is not of much help, since, for instance, a generic conjugation-invariant configuration of $d$ lines in the plane can have at most $\left[\frac{d}{2}\right]$ elliptic nodes. The patchworking construction invented by O. Viro around 1980 for the study of the topology of smooth real algebraic varieties (see, for instance, Appendix to \cite{GLS3}) was later applied to curves and hypersurfaces with singularities (see \cite[Section 2.3]{GLS3}). It allowed one to completely answer the above question \cite{Sh6}. Another solution was later suggested by Pecker \cite{Pe1}, who used explicit parameterizations of real rational curves, obtaining, for instance, the quintic shown in Figure \ref{fig-quintic}(a). Here we demonstrate the patchworking construction, which also  applies efficiently to curves with singularities of other types, while the methods of \cite{Pe1} are restricted to nodal curves only.

The version of the {\em patchworking construction}\index{patchworking!construction} \index{construction!patchworking}  which we need was introduced in \cite{Sh6,Sh12} (see also \cite[Sections 2.3 and 4.5.1]{GLS3}). Let us be given a convex lattice polygon $\Delta\subset\R^2$ and its convex\footnote{A convex subdivision is a subdivision into linearity domains of some convex piecewise linear function defined on the
lattice triangle $T_d=\conv\{(0,0),(d,0),(0,d)\}$.}
  \ subdivision into convex lattice polygons $\Delta_1,...,\Delta_N$. Let $F_1,...,F_N$ be bivariate complex or real polynomials with Newton polygons $\Delta_1,...,\Delta_N$, respectively, such that\\
 (i) the truncations of $F_i$ and $F_j$ on the common side $\sigma$ of $\Delta_i$ and $\Delta_j$ coincide, \\
 (ii) each polynomial $F_i$ is peripherally nondegenerate, i.e., the truncation of $F_i$ on any side of $\Delta_i$ defines a smooth curve in $(\CC^*)^2$, \\
 (iii) each polynomial $F_i$ defines a curve with isolated singularities in $(\CC^*)^2$.

 Denote by $S(F_i)$ the multi-set of topological or  analytic types of the singular points of the curve $F_i(x,y)=0$ in $(\CC^*)^2$.

 Then we orient the adjacency graph of the polygons $\Delta_1,...,\Delta_N$ without oriented cycles and verify the ${\mathcal S}$-transversality condition for each patchworking pattern $(\Delta_i,\partial\Delta_{i,+},F_i)$, where $\partial\Delta_{i,+}$ is the union of the sides of $\Delta_i$ corresponding to the incoming arcs of the adjacency graph and ${\mathcal S}$\index{$s$@$\mathcal S$}  stands for the chosen complex or real topological or analytic equivalence of singularities.

The following {\em patchworking theorem}\index{patchworking!theorem}  for curves says that we can ''glue'' the polynomials $F_i$ together to one polynomial $F$, wich defines a curve with isolated singularities in $(\CC^*)^2$ that
inherit the singularity types from the $F_i$.

\begin{theorem}\label{t-patchworking}
With the above notations and assumptions let $F_1,...,F_N$ with properties {\em (i), (ii), (iii)} be given. Then there exists a polynomial $F(x,y)$ with Newton polygon $\Delta$ such that
$$S(F)=\bigcup_{i=1}^NS(F_i).$$
Moreover, the family of ${\mathcal S}$-equisingular curves defined by polynomials with Newton polygon $\Delta$ is $T$-smooth at $\{F=0\}$.
%
\end{theorem}
We say that a family of curves is {\em ${\mathcal S}$-equisingular}, if the choosen types ${\mathcal S}$ of the singular points of the curves stay locally constant along some section.\\

One of the nicest sides of the patchworking construction\index{construction!patchworking}\index{patchworking!construction} is that it works equally well over the complex and the real fields.
The first example illustrating this feature is as follows.

\begin{theorem}\label{t-realnode}\index{existence!of real nodes}
For any integer $d\ge3$ and nonnegative integers $a,b,c$, the inequality
\begin{equation}a+b+2c\le\frac{(d-1)(d-2)}{2},\label{eq:realnodal}\end{equation}
is necessary and sufficient for the existence of a real plane irreducible curve of degree $d$ having $a$ hyperbolic nodes, $b$ elliptic nodes and $c$ pairs of complex conjugate nodes as its only singularities.
\end{theorem}

\begin{proof} The necessity follows from the genus bound (\ref{genbound}). Thus, we focus on the construction.
We only sketch the proof, which in full detail is presented in \cite{Sh6}.
Namely, we shall prove the theorem in the case $a=c=0$. As noticed in Remark \ref{r-tsmoothness}, it is enough to construct a real rational curve with $\frac{(d-1)(d-2)}{2}$ elliptic nodes. Consider the subdivision of the lattice triangle $T_d=\conv\{(0,0),(d,0),(0,d)\}$ into lattice triangles as shown in Figure \ref{fig2}. Observe that the number of interior integral points in these triangles amounts to $\frac{(d-1)(d-2)}{2}$.

\begin{figure}
\setlength{\unitlength}{1.0cm}
\begin{picture}(7,5)(0,0)
\thinlines

\put(2,0.5){\vector(1,0){4}}\put(2,.5){\vector(0,1){4}}
\dashline{0.2}(2,1)(5,1)\dashline{0.2}(2,2)(4,2)\dashline{0.2}(2,3)(3,3)
\dashline{0.2}(2,1.5)(4.5,1.5)\dashline{0.2}(2,2.5)(3.5,2.5)\dottedline{0.1}(5,1)(5,0.5)

\thicklines

\put(2,0.5){\line(1,0){3.5}}\put(2,0.5){\line(0,1){3.5}}
\put(5.5,0.5){\line(-1,1){3.5}}\put(2,0.5){\line(6,1){3}}
\put(2,1.5){\line(6,-1){3}}
\put(2,1.5){\line(4,1){2}}\put(2,2.5){\line(4,-1){2}}
\put(2,2.5){\line(2,1){1}}\put(2,3.5){\line(2,-1){1}}

\put(5.5,0.6){$d$}\put(4.5,0){$d-1$}\put(1.6,1.4){$2$}\put(1.6,2.4){$4$}
\put(1.6,3.9){$d$}

\end{picture}
\caption{Patchworking of a real plane curve with elliptic nodes}\label{fig2}
\end{figure}
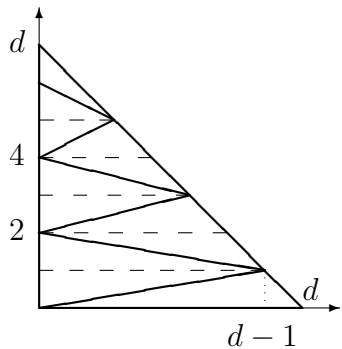

Each tile of the subdivision is a triangle of the form $T=\conv\{(0,0),(0,2),(m,1)\}$ (up to an automorphism of the lattice $\Z^2$). We claim that the real polynomial
\begin{equation}F(x,y)=y^2-2yP_m(x+\lambda)+1,\label{eq:chebyshev}\end{equation}
where $P_m(x)=\cos(m\arccos x)$ is the $m$-th Chebyshev polynomial and $\lambda$ is a generic real number, has Newton polygon $T$ and defines a real plane curve with $m-1$ elliptic nodes in $(\R^*)^2$. Note that $m-1$ is the number of interior integral points in $T$. To prove the claim it is enough to rewrite the equation of the curve $F(x,y)=0$ in the form
$$y=P_m(x+\lambda)\pm\sqrt{P_m(x+\lambda)^2-1}$$
and recall that $P_m$ has $m-1$ extrema:  $\left[\frac{m}{2}\right]$ minima on the level $-1$ and $\left[\frac{m-1}{2}\right]$ maxima on the level $1$.

Now we associate with each triangle $T^{(i)}$ of the subdivision a polynomial $F^{i)}$ with Newton triangle $T^{(i)}$ which is obtained from a polynomial like (\ref{eq:chebyshev}) by the coordinate change matching an appropriate automorphism of the lattice $\Z^2$. A further transformation $F^{(i)}(x,y)\mapsto\alpha_i F^{(i)}(\beta_i x,\gamma_i y)$ with suitable positive $\alpha_i,\beta_i,\gamma_i\in\R$ equates the truncations of each pair of the neighboring polynomials $F^{(i)}$, $F^{(i+1)}$ on the common side of $T^{(i)}$ and $T^{(i+1)}$. To complete the proof, we apply Theorem \ref{t-patchworking}, observing that, in the nodal case, every patchworking pattern $(T^{(i)},\partial_+T^{(i)},F^{(i)})$ is ${\mathcal S}$-transversal (${\mathcal S}$ being the topological or analytic equivalence of singularities), see \cite[Theorem 4.2]{Sh12} or \cite[Corollary 4.5.3]{GLS3}.
\end{proof}

\subsection{Plane curves with nodes and cusps}\label{ssec:2.2}

%
%
%


Questions concerning the number of nodes and cusps on a plane curve of a given degree are classical and highly nontrivial as compared to the purely nodal case, in particular over the reals. No complete answer is known in general, neither in the complex case, nor in the real one.

The general restrictions for their existence from Section \ref{sec.arb:restr}, such as the genus formula, the Pl\"ucker formulas, and the bounds by Hirzebruch-Ivinskis  \cite{HF,Iv} and Langer \cite{Lan} take a special simple form for curves with nodes and cusps. We compare the asymptotics of the bounds by Hirzebruch-Ivinskis and Langer if the degree goes to infinity.

In the second part of this section we report on the state of knowledge on curves with many cusps, complex as well as real. Special attention will be given to small degrees, with precise references to their construction.

The last part is devoted to the patchworking construction, which  provides an asymptotically complete answer if we restrict to real and complex plane curves with nodes and cusp belonging to $T$-smooth equisingular families.

 For the results of this section see also
\cite[Sections 4.2.2.2, Section 4.5.2.1]{GLS3} (for restrictions) and \cite[Sections 4.5.2.2, 4.5.2.3]{GLS3} (for constructions).

In the whole section we consider complex or real curves which are irreducible over the complex numbers.\\

\bigskip\noindent
{\bf Restrictions for the existence}\index{restrictions!for existence!of nodes and cusps}
\medskip

Let $C$ be a curve with only nodes and cusps as singularities.
For a node we have $\del = \mu=1$, $\kappa = 2$ and for a cusp $\del =1, \, \mu = 2, \ \kappa = 3$.
Hence
the formulas (\ref{mubound}), resp.  (\ref{genusformula}), give as a necessary condition for the existence of an irreducible, resp. not necessarily irreducible, curve with $n$ nodes and $k$ cusps the estimates
\begin{eqnarray}
2 n + 2 k & \leq & (d-1)(d-2),  \text{ resp.}\nonumber\\
n + 2 k & \leq & (d-1)^2.\nonumber
\end{eqnarray}
From the Pl\"ucker formulas\index{Pl\"ucker!formula} (\ref{e304}) and (\ref{eq:hess}) we get the necessary conditions for irreducible curves, i.e. for the non-emptiness of $V_d^{irr}(nA_1+kA_2)$,
 (originally due to Lefschetz \cite{Lef}, combined with the fact that $d^*(d^*-1)\ge d$ and formula (\ref{eq:hesssh}))
\begin{eqnarray}
2 n + 3 k & \leq & d^2\!\!\:-d-\sqrt{d}\,, \nonumber\\
6 n + 8 k & \leq & 3d^2\!\!\:-6d\,.\label{nchess}
\end{eqnarray}.

\smallskip
Better bounds from Hirzebruch-Ivinskis, Langer, and spectral estimates
are obtained as a consequence of deep results in algebraic geometry.
The general Hirzebruch-Ivinskis inequality  (\ref{eq:HI})
reads for a curve of an even degree $d\ge10$ with  $n$ nodes and $k$ cusps as  (cf. (\ref{eq:Iv1}))
\begin{equation}\frac{9}{8}n+2k\,\le\, \frac{5}{8}d^2\!-\frac{3}{4}d\quad\text{for all}\ d\ \text{even},\ d\ge6\ \label{eq:Iv2}\end{equation}

Langer's inequality (\ref{elan2}) applies to all degrees $d\ge10$, and in this range it is always better then (\ref{eq:Iv2}). Hirzebruch-Ivinskis (HI)
compared to Langer (L) gives:
\begin{align}\text{(HI):}&\qquad 0.5625n+k\le0.3125d^2-0.375d,\nonumber\\
\text{(L):}&\qquad 0.5821n+k\le0.3091d^2-0.3453d.\nonumber\end{align}
In particular, Langer's inequality implies that the \emph{maximal number of cusps} $k_{\max}(d)$ on a curve of degree $d$ satisfies
\begin{equation}\lim_{d\to\infty}\sup\frac{k_{\max}(d)}{d^2}\le\frac{125+\sqrt{73}}{432}=0.3091...\ .
\label{eq:langer}\end{equation}
We also mention the spectral bound \cite[Theorem, page 164]{Va}
\begin{equation}\frac{1}{2}n+k\le\frac{1}{2}\cdot\left(\begin{array}{l} \text{the number of integral points}\ (i,j)\ \text{satisfying}\\
0<i,j<d,\quad \frac{d}{6}<i+j\le\frac{7}{6}d\end{array}\right)\label{eq:var}\end{equation}
$$=\frac{23}{72}d^2+O(d)\approx0.3194d^2+O(d),$$
which is asymptotically weaker than the Hirzebruch-Ivinskis' and Langer's bounds.\\

\medskip\noindent{\bf Curves with a large number of cusps}
\medskip

The problem of existence of plane curves with a large number of cusps attracted a special attention, due to the fact that the maximal possible number of cusps in general is not known yet. We shortly describe here several constructions.

Following Zariski \cite{Za2}, \index{construction!Zariski's} consider curves $C_r^{(1)}\subset\PP^2$ of degree $d=d_r=6r$, $r\ge1$, given by
$F^2+G^3=0$, where $F,G\in\CC[x,y,x]$ are generic homogeneous polynomials of degree $3r$ and $2r$, respectively.
The curves $F=0$ and $G=0$ then intersect transversally at $6r^2$ distinct points, and each of these intersection points is an ordinary cusp of the curve $C^{(1)}_r$. The total number of cusps is $6r^2=\frac{d^2}{6}$, which is far from the upper bounds discussed above. However, choosing appropriate real $F$ and $G$, one can obtain a real irreducible curve $C^{(1)}_r$ of degree $d=6r$ with $\frac{d^2}{6}$ real cusps.

Ivinskis' construction \cite{Iv} \index{construction!Ivinskis'} has provided a bigger number of cusps. Namely, he started with the sextic curve
$F(x,y,z)=0$ having $9$ cusps in the torus $(\CC^*)^2$ and considered the series of curves $C_r^{(2)}=\{F(x^r,y^r,z^r=0\}$, $r\ge1$, of degree $d=d_r=6r$. Since the substitution of $(x^r,y^r,z^r)$ for $(x,y,z)$ defines an $r^2$-sheeted covering of the torus $(\CC^*)^2$, we obtain that $C^{(2)}_r$ has $9r^2=\frac{d^2}{4}$ cusps.
This number of cusps is closer to the upper bounds, but the number of real cusps of $C^{(2)}_r$ is at most $12$.

A new idea was suggested by Hirano \cite{HA}: \index{construction!Hirano's} she used a sequence of coverings defined by the substitution of $(x^3,y^3,z^3)$ for $(x,y,z)$, choosing each time the coordinate system with axes tangent to the current curve, and noticing that each tangent point like that lifts to three ordinary cusps. In particular, starting with the sextic curve having $9$ cusps, she produced the sequence of curves $C^{(3)}_k$, $k\ge1$, of degree $d=d_k=2\cdot 3^k$, having
$$s_k=\frac{9}{8}(9^k-1)$$
cusps. The limit ratio of the number of cusps by the square of the degree equals here
$$\lim_{k\to\infty}\frac{s_k}{d_k^2}=\frac{9}{32}=0.28125,$$
which is closer to the Langer's bound (\ref{eq:langer}). Moreover, this was the first example of an equisingular family having a negative expected dimension
$$\frac{d_k(d_k+3)}{2}-2s_k\sim-\frac{d_k^2}{16}\quad\text{as}\ k\to\infty,$$
implying that the cusps impose dependent conditions on the space of curves of degree $d$ and that this $ESF$ is not $T$-smooth.

The construction by Hirano was later refined by Kulikov \cite{Ku} and then by Calabri et al. \cite[Theorem 6]{CPS}, who found a sequence of curves $C^{(4)}_k$ of degree $d_k=108\cdot 9^k$, $k\ge1$, having at least
$$s_k=\frac{69309}{20}9^{2k}-27\cdot9^k-\frac{9}{20}$$ cusps, which yields the (so far) best known limit ratio
$$\lim_{k\to\infty}\frac{s_k}{d_k^2}=\frac{2567}{8640}=0.2971...\ .$$

Notice also that the constructions
by Hirano, Kulikov, and Calabri et al. give a rather small number of real cusps.
\medskip

For small degrees $d$, the known upper bounds often coincide with the known maximal number $s_{\max}(d)$ of cusps of a plane curve of degree $d$ or differ by $1$. We present the results in the following table with an additional information on the known maximal number $s_{\max,\R}(d)$ of real cusps on a real curve of degree $d$.\index{bound!small degrees}
\begin{center}
\begin{tabular}{|l|c|c|c|c|c|c|c|c|c|c|c|c|}
\hline
degree &\makebox[0.35cm] 3 &\makebox[0.35cm] 4
&\makebox[0.35cm] 5 &\makebox[0.35cm] 6 &\makebox[0.35cm] 7
&\makebox[0.35cm] 8 &\makebox[0.35cm] 9 &\makebox[0.35cm]{10}
&\makebox[0.35cm]{11} &\makebox[0.35cm]{12}\\
\hline
upper bound & 1 & 3 & 5 & 9 & 10 & 15 & 21 & 26 & 31 & 40 \\
\hline $s_{\max}$ & 1 & 3 & 5 & 9 & 10 & 15 & 20 & 26 & 30 & 39 \\
\hline $s_{\max,\R}$ & 1 & 3 & 5 & 7 & 10 & 14 & 14 & 18 & 23 & 28 \\
\hline
\end{tabular}
\end{center}
\medskip

The upper bounds for $3\le d\le6$ follow from the second Pl\"ucker formula in the form (\ref{nchess}). The upper bounds for $7\le d\le9$ were proved by Zariski \cite[Pages 221, 222]{Za2}: he showed that for the $d$-multiple cover of the plane ramified along a plane curve of degree $d$ 
with only nodes and cusps, the irregularity vanishes \cite[page 213]{Za2}, and then he derived that the family of curves of degree $d-3-\left[\frac{d-1}{6}\right]$ passing through the 
cusps of the given curve was unobstructed; hence, the number of cusps does not exceed
$\frac{(d-m)(d-3-m)}{2}+1$, where $m=\left[\frac{d-1}{6}\right]$ (cf. \cite[Formula (18) on page 221]{Za2}).
The bounds for $9\le d\le 10$ follow also from the spectral estimate (\ref{eq:var}). The bound $31$ for $d=11$ follows from the semicontinuity of the spectrum for lower deformations of quasihomogeneous singularities \cite[Theorem in page 1294]{Var1} applied to the open interval $\left(\frac{2}{11},\frac{13}{11}\right)$ (different from that in (\ref{eq:var})). At last, the bound $40$ for $d=12$ follows from the Hirzebruch-Ivinskis estimate (\ref{eq:Iv2}).

The cubic, quartic, quintic, sextic, and septic with the indicated number of cusps are classically known, for example, the quartic is dual to the nodal cubic, while the sextic is dual to a smooth cubic, the quintic was known to Segre (an explicit construction can be found in \cite{Gu}, see also Figure \ref{fig-quintic}(b)). The maximal known cuspidal curves for $d=10$ and $12$ were constructed by Hirano \cite{HA}. The maximal known cuspidal curves for $d=8$, $9$, and $11$ were constructed via patchworking respectively in \cite[Theorem 4.3]{Sh12} and \cite[Appendix]{CPS},
\cite[Section 4.5.2.3]{GLS3}. The maximal cuspidal septic was constructed by Zariski \cite[Page 222]{Za2} (see also \cite{Koe}). We comment on this result, which nicely combines duality of curves with the classical result on deformation of curves: the nodes and cusps of an irreducible plane curve of degree $d$ can be independently deformed in a prescribed way or preserved as long as the number of cusps is less than $3d$ \cite{Segre,Segre2}. Zariski starts with the sextic having $9$ cusps, deforms it into the sextic with $7$ cusps and one node, takes its dual which is a curve of degree $7$ with $10$ cusps and $3$ nodes and, finally, smoothes out the nodes.

For $d\le5$, there are real maximal cuspidal curves with only real cusps: such a quartic is dual to the cubic with an elliptic node, the cuspidal quintic constructed in \cite{Gu} and shown in Figure \ref{fig-quintic}(b) is real and has only real cusps. The known maximal numbers of real cusps for degrees $6\le d\le 8$ were reached in \cite{IS}. It is interesting that $s_{\max,\R}(6)=7$ is the actual maximum for real sextics as shown in \cite{IS}: the absence of a real sextic with $8$ real cusps was derived from a delicate analysis of the moduli space of real K3 surfaces obtained as double covers of the plane ramified along a real plane sextic curve.
The values of $s_{\max,\R}$ for $9\le d\le12$ are borrowed from Theorem \ref{t-nodecusp} below.\\

\medskip\noindent
{\bf Patchworking curves with nodes and cusps.}
\medskip

The patchworking construction \index{construction!patchworking}\index{patchworking!construction}gives an asymptotically proper answer to the existence problem for real and complex plane curves with nodes and cusps.
If we restrict our attention to real and complex plane curves with nodes and cusp which belong to $T$-smooth equisingular families, then this construction, in view of (\ref{regex}), provides an asymptotically complete answer (see \cite[Theorems 2.2, 3.3, and 4.1]{Sh12} and \cite[Section 4.5.2.2]{GLS3}):

\begin{theorem}
\index{existence!of nodes and cusps}
For any non-negative integers $d$, $n$, $k$ such that
\begin{equation}
n+2k\,\le\,\frac{d^2\!-4d+6}{2}\ ,\label{e32}
\end{equation}
there exists a (complex)
plane irreducible curve with $n$ nodes and $k$ cusps as only singularities. Moreover, the result is
asymptotically $T$-smooth optimal, i.e., up to linear terms in $d$ no more nodes and cusps are possible on a  curve belonging to a $T$-smooth $ESF$.
\end{theorem}

\begin{theorem}\label{t-nodecusp}\index{existence!of real nodes and cusps}
(1) For any $d\ge3$ and any positive integer $c$ such that
\begin{equation}c\le\frac{d^2-3d+4}{4}\ ,\label{eq:realcusp}\end{equation}
there exists a real plane curve of degree $d$ with $c$ real cusp as its only singularities.

(2) There exists a linear in $d$ polynomial $\psi(d)$ such that, for any $d\ge3$ and nonnegative integers $n_h,n_e,n_{im},c_{re},c_{im}$ with
\begin{equation}n_h+n_e+2n_{im}+2c_{re}+4c_{im}\le\frac{d^2}{2}+\psi(d),\label{eq:realnodescusps}\end{equation}
there is a real plane curve of degree $d$ having $n_h$ hyperbolic nodes, $n_e$ elliptic nodes, $n_{im}$ pairs of complex conjugate nodes, $c_{re}$ real cusps, and $c_{im}$ pairs of complex conjugate cups as its only singularities.

\smallskip
\noindent Moreover, these curves correspond to $T$-smooth
germs of the respective equisingular families of curves with
cusps in (1) resp. with nodes and cusps in (2) and the bounds in (\ref{eq:realcusp}) and  (\ref{eq:realnodescusps})
are asymptotically optimal w.r.t. $T$-smooth $ESF$.
\end{theorem}

\begin{proof}
We prove here the part (1) of Theorem \ref{t-nodecusp}, referring to the references above for the rest. Consider the subdivision of the lattice triangle $T_d=\conv\{(0,0),(d,0),(0,d)\}$ into the lattice quadrangles an triangles depicted in Figure \ref{fig4}(a). It is not difficult to show that the subdivision is convex.

\begin{figure}
\begin{center}
\setlength{\unitlength}{0.6cm}
\begin{picture}(16,12)
\thicklines
\put(0,1){\vector(1,0){11}}
\put(0,1){\vector(0,1){11}}
\put(0,11){\line(1,-1){10}}
\put(0,8){\line(1,-1){1}}
\put(0,6){\line(1,-1){1}}
\put(0,4){\line(1,-1){1}}
\put(0,2){\line(1,-1){1}}
\put(2,3){\line(1,-1){1}}
\put(2,5){\line(1,-1){1}}
\put(2,7){\line(1,-1){1}}
\put(4,6){\line(1,-1){1}}
\put(4,4){\line(1,-1){1}}
\put(4,2){\line(1,-1){1}}
\put(6,3){\line(1,-1){1}}
\put(1,9){\line(1,0){1}}
\put(1,7){\line(1,0){1}}
\put(1,5){\line(1,0){1}}
\put(1,3){\line(1,0){1}}
\put(3,2){\line(1,0){1}}
\put(3,4){\line(1,0){1}}
\put(3,6){\line(1,0){1}}
\put(5,5){\line(1,0){1}}
\put(5,3){\line(1,0){1}}
\put(7,2){\line(1,0){1}}
\put(0,2){\line(1,1){1}}
\put(0,4){\line(1,1){1}}
\put(0,6){\line(1,1){1}}
\put(0,8){\line(1,1){1}}
\put(2,7){\line(1,1){1}}
\put(2,5){\line(1,1){1}}
\put(2,3){\line(1,1){1}}
\put(2,1){\line(1,1){1}}
\put(4,2){\line(1,1){1}}
\put(4,4){\line(1,1){1}}
\put(6,3){\line(1,1){1}}
\put(6,1){\line(1,1){1}}
\put(0,10){\line(1,0){1}}
\put(0,10){\line(1,-1){1}}
\put(4,6){\line(0,1){1}}
\put(4,6){\line(1,0){1}}
\put(8,2){\line(0,1){1}}
\put(8,2){\line(1,0){1}}
\put(8,2){\line(1,-1){1}}
\put(7,2){\line(1,-1){1}}
\put(3,2){\line(1,-1){1}}
\thinlines
\put(1,3){\line(1,-2){1}}
\put(1,5){\line(1,-2){1}}
\put(1,7){\line(1,-2){1}}
\put(1,9){\line(1,-2){1}}
\put(3,8){\line(1,-2){1}}
\put(3,6){\line(1,-2){1}}
\put(3,4){\line(1,-2){1}}
\put(5,3){\line(1,-2){1}}
\put(5,5){\line(1,-2){1}}
\put(7,4){\line(1,-2){1}}
\dashline{0.2}(0,9)(1,9)
\dashline{0.2}(0,8)(3,8)
\dashline{0.2}(0,7)(1,7)
\dashline{0.2}(2,7)(4,7)
\dashline{0.2}(0,6)(3,6)
\dashline{0.2}(0,5)(1,5)
\dashline{0.2}(2,5)(5,5)
\dashline{0.2}(0,4)(3,4)
\dashline{0.2}(4,4)(7,4)
\dashline{0.2}(0,3)(1,3)
\dashline{0.2}(2,3)(5,3)
\dashline{0.2}(6,3)(8,3)
\dashline{0.2}(0,2)(3,2)
\dashline{0.2}(4,2)(7,2)
\dashline{0.2}(1,1)(1,10)
\dashline{0.2}(2,1)(2,9)
\dashline{0.2}(3,1)(3,8)
\dashline{0.2}(4,1)(4,6)
\dashline{0.2}(5,1)(5,6)
\dashline{0.2}(6,1)(6,5)
\dashline{0.2}(7,1)(7,4)
\dashline{0.2}(8,1)(8,2)
\dashline{0.2}(9,1)(9,2)
\put(9.9,0.4){$d$}
\put(-0.5,10.9){$d$}
\put(6,-1){\text{\rm (a)}}

\thicklines
\put(13,1){\vector(1,0){4}}
\put(13,1){\vector(0,1){4}}
\put(13,7){\vector(1,0){4}}
\put(13,7){\vector(0,1){4}}
\put(13,2){\line(1,-1){1}}
\put(13,2){\line(1,1){1}}
\put(13,9){\line(1,-2){1}}
\put(13,9){\line(1,0){1}}
\put(14,3){\line(1,-2){1}}
\put(14,7){\line(1,1){1}}
\put(14,9){\line(1,-1){1}}
\thinlines
\dashline{0.2}(13,2)(15,2)
\dashline{0.2}(13,3)(14,3)
\dashline{0.2}(14,1)(14,3)
\dashline{0.2}(15,1)(15,2)
\dashline{0.2}(13,8)(15,8)
\dashline{0.2}(14,7)(14,9)
\dashline{0.2}(15,7)(15,8)
\dashline{0.2}(13,4)(16,1)
\dashline{0.2}(13,10)(14,9)
\dashline{0.2}(15,8)(16,7)
\put(15,4){$Q_1$}
\put(15,10){$Q_2$}
\put(15,-1){\text{\rm (b)}}

\end{picture}

\vspace{20pt}
\caption{Construction of real curves with real cusps}\label{fig4}
\end{center}
\end{figure}
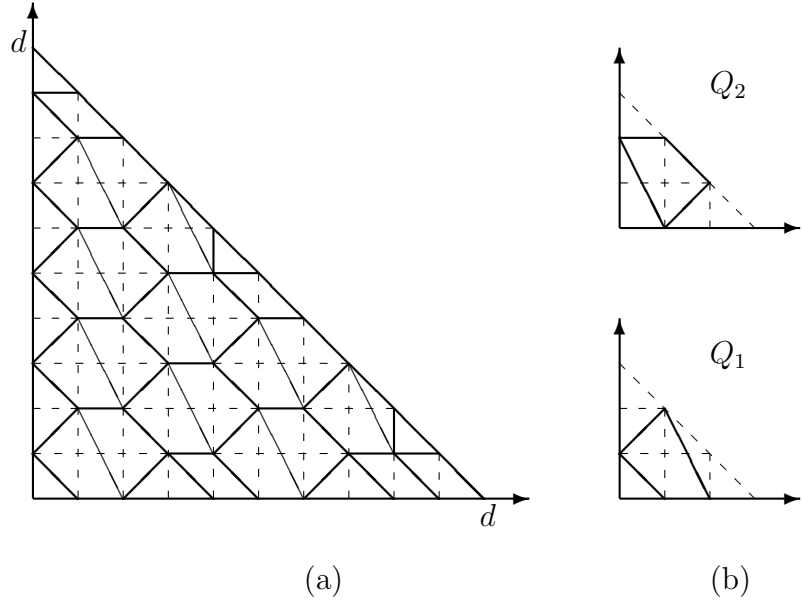
It is easy to see that, for any given $\alpha,\beta,\gamma\in\R^*$, there exist real polynomials
\begin{align}F_1(x,y)&=\alpha y+\beta x+\gamma x^2+a_{11}xy+a_{12}xy^2,\nonumber\\
F_2(x,y)&=\alpha y^2+\beta x+\gamma x^2y+b_{11}xy+b_{12}xy^2,\nonumber\end{align}
with Newton quadrangles $Q_1,Q_2$ (see Figure \ref{fig4}(b)) which define curves with a real cusp in $(\R^*)^2$.
Both curves coincide with the real cuspidal cubic tangent to the coordinate axes in an appropriate way. Thus, we can associate compatible real polynomials with each tile of the subdivision so that the polynomials for the translates of $Q_1$ and $Q_2$ will define real curves with a real cusp in $(\R^*)^2$. Orient the adjacency graph of the subdivision so that, for each pattern $(\Delta_i,\partial_+\Delta_i,G_i)$, the part of the boundary $\partial_+\Delta_i$ will consist of the two lower sides of each translate of $Q_1$ and $Q_2$. The lattice length of the rest of the boundary is $2$, which is greater than $1$, the number of cusps in $(\CC^*)^2$, which finally yields the transversality of each pattern (see \cite[Theorem 4.1(1)]{Sh12} or \cite[Proposition 4.5.2]{GLS3}). By Theorem \ref{t-patchworking} there exists a real curve of degree $d$ with real cusps as its only singularities, where the number of cusps equals $\left[\frac{d^2-3d+4}{4}\right]$, the number of translates of $Q_1$ and $Q_2$ in the subdivision. Since the resulting curve belongs to a $T$-smooth equisingular family, it can be deformed with smoothing out prescribed cusps.
\end{proof}

We note that the inequalities (\ref{e32}), (\ref{eq:realcusp}) and (\ref{eq:realnodescusps}) differ, w.r.t. the number of cusps, from the restriction of the genus bound (\ref{genbound}) by a factor $\frac{1}{2}$ at $d^2$ and a term linear in $d$, showing that the result is asymptotically proper.

\section{Plane curves with arbitrary singularities}\label{sec:3}

In this section we discuss curves with special singularities (ordinary multiple points, simple singularities) as well as curves with arbitrary singularities up to topological or  analytic equivalence.
Moreover, the bounds based on the Bogomolov-Miyaoka-Yau inequality are mainly restricted to simple singularities and semi-quasihomogeneous singularities (we discuss this in more detail in Section \ref{ssec:3.2}). Thus, in general we are left only with the genus bound, the Pl\"ucker bounds, and the spectral bound.

\subsection{Curves of small degrees}\label{ssec:3.1}

%
%
%

For degrees $\le6$ the possible collections of singularities of irreducible complex plane curves are classified.

The fact that an irreducible real or complex cubic may have either a node $A_1$ or a cusp $A_2$ was known already to Newton \cite{N1,N2,N3} (for a modern treatment of Newton's study, see \cite{KW}).

Collections of singularities of complex irreducible quartic curves can easily be classified by manipulating  with equations or by using quadratic Cremona transformations, and all this has been known to algebraic geometers of the 19-th century. Moreover, it can be shown that each collection of singularities defines a smooth irreducible subvariety of expected dimension in the space $\PP^{14}$ of plane projective quartics (see, for instance, \cite{BG} and \cite{W1}). The classification of real singular quartic curves was completed in \cite{Gu1} (see also \cite{KW1} for more details as well as for the classification of real singular affine quartics)\footnote{Both papers contain much more material: they classify all possible dispositions of singular points on the real point set of the quartic curve.}.

Still, the classification of collections of singularities of irreducible plane quintic curves can be reached by elementary methods. The complete classification of singularities of plane quintics together with the statement that each collection of singularities defines a smooth irreducible equisingular family of expected dimension can be found in \cite{De1} and \cite[Section 7.3]{De2} (se also \cite{Gu} and \cite{W2} for interesting particular examples of singular quintics).

The case of plane sextic curves is the first non-elementary one. The study of plane sextics heavily relies of a thorough investigation of K3 surfaces appearing as double covers of the plane ramified along the considered sextic curve. On the other hand, it reveals a highly interesting new phenomenon - the existence of the so-called {\em Zariski pairs}\index{pair!Zariski}, i.e., pairs of curves having the same collection of singularities, but belonging to different components of the equisingular family. The first example presenting two sextic curves with $6$ ordinary cusps that have non-homeomorphic complements in the projective plane was found by Zariski \cite[page 214]{Za2}.
The complete classification of collections of singularities (up to topological equivalence) can be found in \cite{AG} and \cite[Section 7.2]{De2} (various particular cases have been investigated in \cite{U1,U2,U3}). The classiication of real singularities of real sextics is not completed yet (for the case of complex and real cusps, see Section \ref{ssec:2.2} above).

\subsection{Curves with simple, ordinary, and semi-quasihomogeneous singularities}\label{ssec:3.2}

%
%
%

In the present section we construct equisingular families of curves with many simple resp. ordinary resp. semi-quasihomogeneous singularities  and compare their number with the known necessary bounds. Again, by means of the patchworking construction we are able construct families such that the number of singularities is asymptotically optimal resp. proper. The results presented in Theorem  \ref{t-ordinary1},  \ref{t-sqh} and \ref{t-squ1} are new.

\medskip
\noindent
{\bf Curves with simple singularities.} The patchworking construction, essentially used in Section \ref{sec:2} for the construction of curves with real and complex nodes and cusps, works equally well in the case of arbitrary simple singularities $A_n$, $n\ge1$, $D_n$, $n\ge4$, $E_n$, $n=6,7,8$. The results of \cite{ShW,Wes1} on the existence of plane complex curves with simple singularities can be summarized in the following statement (cf. \cite[Theorem 4.5.5]{GLS3})

\begin{theorem}\label{t-simple1}\index{existence!of simple singularities}
(1) For any simple singularity $S$, there exists a linear polynomial
$\varphi_S(d)$ such that the inequality
\begin{equation}k\mu(S)\le\frac{d^2}{2}+\varphi_S(d)\label{eq:simple1}\end{equation}
is sufficient for the existence of an irreducible complex plane curve of degree $d$ having $k$ isolated
singular points of type $S$ as its only singularities and belonging to a $T$-smooth ESF.

(2) Furthermore, for any integer $m\ge1$, there exists a linear polynomial $\psi^{simple}_m(d)$ such that the inequality
\begin{equation}\sum_{i=1}^r\mu(S_i)\le\frac{d^2}{2}+\psi^{simple}_m(d),\label{eq:simple2}\end{equation}
with arbitrary $r$ and simple singularities $S_1,...,S_r$ with Milnor numbers $\le m$, is sufficient for the existence of an irreducible complex plane curve of degree $d$ having $r$ isolated singular points of types $S_1$, ..., $S_r$, respectively, as its only singularities and belonging to a $T$-smooth ESF.
\end{theorem}

Notice that this existence statement is \emph{asymptotically optimal} as long as we consider curves belonging to $T$-smooth $ESF$. Indeed, the codimension of a $T$-smooth $ESF$ in the considered cases equals the left-hand side in (\ref{eq:simple1}) or (\ref{eq:simple2}) (see \cite[Corollary 6.3(ii)]{GK} and \cite[Theorem 2.2.40]{GLS3}), and hence satisfies
$$k\mu(S)\le\frac{d(d+3)}{2},\quad\text{resp.}\quad\sum_{i=1}^r\mu(S_i)\le\frac{d(d+3)}{2}.$$
In each case, the difference between the bounds in the necessary and sufficient conditions is linear in $d$. For the proof we refer to \cite{ShW,Wes1} (see also the proof of Theorem \ref{t-ordinary1} below).

The existence of real curves with simple singularities was analyzed in \cite{Wes} along the same lines, though the argument was incomplete: in particular, Lemma 3.1 in \cite{Wes} is wrong as pointed out by E. Brugall\'e. So, in general the problem over the reals remains open.

\smallskip
While the patchworking construction basically resolves the existence problem for curves with simple singularities that impose independent conditions on the coefficients of the defining equation, there are examples of curves with extremely many simple singularities imposing thereby dependent conditions on the curve.

So, the construction invented by Hirano \cite{HA} (which we discussed in connection to plane curves with large number of cusps, Section \ref{ssec:2.2}) applies well also to more complicated cusps $A_n$ (see \cite[Theorem 2]{HA}. Namely, one starts with a smooth conic and, in each step, chooses axes tangent to the current curve and substitutes $(x^{n+1},y^{n+1},z^{n+1})$ for $(x,y,z)$, which results in the following sequence of singular curves:

\begin{theorem}\label{t-An}\index{existence!of $A_n$ singularities}
For any even $n\ge4$, there exists a sequence of irreducible complex plane curves $C_k$, $k=1,2,...$, of degree $d_k=2(n+1)^k$ having
$$s_k=\frac{3(n+1)((n+1)^{2k}-1)}{n(n+2)}$$
singular points of type $A_n$ as their only singularities.
\end{theorem}

\begin{remark}\label{r-hirano}{\em
In fact, Hirano's construction works for odd $n\ge3$ as well with the above formulas for the degree and for the number of $A_n$ singularities, but the curves appear to be reducible.}
\end{remark}

Since a singularity $A_n$ imposes in general $\tau(A_n)=\mu(A_n)=n$ conditions, we obtain that the total number of conditions compared with the dimension of the space of curves of the given degree reveals the following asymptotics:
\begin{equation}\lim_{k\to\infty}\frac{ns_k}{d_k(d_k+3)/2}=\frac{3}{2}-\frac{3}{2(n+2)}
>1\quad\text{for all}\ n\ge2.
\label{eq:hirano}\end{equation}
Thus, the $A_n$ impose dependent conditions and the corresponding equisingular stratum is not $T$-smooth. For instance, in general, we cannot decide whether there exist curves of degree $d_k$ with any number $m<s_k$ of singularities of type $A_n$.
\medskip

Other series of extremal examples exhibit plane curves with one singularity $A_m$ with $m=m(d)$ as large as possible for a given degree $d$. More precisely, each series consists of plane curves of degrees $d\to \infty$ with one $A_{m(d)}$ singularity:
Gusein-Zade and Nekhorochev \cite[Proposition 2]{GZN} constructed a series of curves for which
$$\lim_{d\to\infty}\frac{m(d)}{d(d+3)/2}=\frac{15}{14}=1.0714...\ .$$
Later Cassou-Nogues and Luengo \cite{CNL} obtained another sequence of curves with
$$\lim_{d\to\infty}\frac{m(d)}{d(d+3)/2}=8-4\sqrt{3}=1.0717...\ .$$
The best known result is due to Orevkov \cite[Section 4]{Ore}:

\begin{theorem}\label{t-GZN}\index{existence!of $A_n$ singularities}
There exists a sequence of plane curves of degrees $d_k\to\infty$ as $k\to\infty$ having a singular point of type
$A_{m_k}$ such that
$$\lim_{k\to\infty}\frac{m_k}{d_k(d_k+3)/2}=\frac{7}{6}=1.1666...\ .$$
\end{theorem}

The ratios of the number of imposed conditions to the dimension of the space of curves of a given degree in all these examples appears to be $>1$; hence, we again observe a non-$T$-smooth equisingular family.
Also, in general, we cannot decide whether there exist curves of degree $d_k$ with an $A_r$ singularity for all $r<m_k$.
\medskip

It is worth to compare the extremal examples by Hirano and Orevkov with the known restrictions. The genus bound (\ref{genbound}) and Pl\"ucker formulas (\ref {plucker}), (\ref{eq:hess}) yield weaker bounds than the Hirzebruch-Ivinskis' and the spectral ones.  Namely, the Hirzebruch-Ivinskis bound (\ref{eq:HI}) combined with the first formula in (\ref{eq:Iv1}) implies that
the limit ratio of the total Milnor number to the dimension of the space of curves of the given degree does not exceed \begin{align}\frac{5(n+1)}{3(n+2)}&\quad\text{in the case of Hirano},\nonumber\\
\frac{5}{3}\;\;\;\;&\quad\text{in the case of Orevkov}.\nonumber\end{align}
By (\ref{eq:spectrum}) the spectra of the singularity $A_n$ and of an ordinary $d$-fold singularity are
$$\left\{\frac{k}{n+1}-\frac{1}{2}\ :\ 1\le k\le n\right\},\quad\left\{\frac{i+j}{d}-1\ :\ 1\le i,j\le d-1\right\},$$
respectively. Applying the semicontinuity of the spectrum on the interval $\qquad$ $(\frac{n}{n+1}-\frac{3}{2},\frac{n}{n+1}-\frac{1}{2}]$, one can easily derive an upper bound of the total Milnor number for the existence of a curve of fixed degree $d$ with $s$ singular points of type $A_n$:
\begin{equation}sn\le\#\left\{(i,j)\in\Z^2\ \big|\ 1\le i,j<d,\ \frac{n}{n+1}-\frac{1}{2}<\frac{i+j}{d}\le\frac{n}{n+1}+\frac{1}{2}\right\}.\label{e111}\end{equation} If we fix $n\ge 2$ and let $d\to\infty$, then we obtain
$$\lim_{d\to\infty}\sup\frac{sn}{d(d+3)/2}\le\frac{3}{2}-\frac{2}{(n+1)^2},$$
which is comparable with the limit ratio in examples of Hirano (\ref{eq:hirano}), showing that, for large $n$ the spectral bound is almost sharp.

If we let $s=1$ and $d\to\infty$ in (\ref{e111}), we will obtain
$$\lim_{d\to\infty}\sup\frac{n}{d(d+3)/2}=\frac{3}{2},$$
which differs from the asymptotical ratio attained in Orevkov's examples, Theorem \ref{t-GZN}, leaving open the question on the sharpness of the spectral bound in the case of one singularity $A_n$.


\medskip\noindent
{\bf Curves with ordinary multiple singular points.}
By an \emph{ordinary (multiple) singular point} 
we understand a singularity consisting of several smooth local branches intersecting  each other transversally. It happens that the patchworking construction \index{construction!patchworking}\index{patchworking!construction}provides an asymptotically optimal existence condition for curves with ordinary multiple points as well, which, moreover, completely covers both the complex and the real case.

\begin{theorem}\label{t-ordinary1} \index{existence!of ordinary singularities}
For a fixed positive integer $m$ the following holds:

(1) There exists a linear polynomial $\psi_{m,fix}^{ordinary}(d)$ such that, for an arbitrary sequence of integers $r_2,...,r_m\ge0$, the inequality
\begin{equation}\sum_{i=2}^m\frac{i(i+1)}{2}r_i\le\frac{d^2}{2}+\psi_{m,fix}^{ordinary}(d)\label{eq:ordinary2}\end{equation}
is sufficient for the existence of an irreducible complex plane curve of degree $d$ in a T-smooth equisingular family, having $r_i$ ordinary singular points of multiplicity $i=2,...,m$ as its only singularities, all of them in general position.

(2) Furthermore, the same inequality (\ref{eq:ordinary2}) is sufficient for the existence of a real plane irreducible curve of degree $d$ in a T-smooth equisingular family, having the given collection of ordinary multiple points in conjugation-invariant general position, when we prescribe the numbers of pairs of imaginary ordinary singular points for each multiplicity $2,...,m$ and prescribe the number of real local branches for each real ordinary singular point.
\end{theorem}

Note that the necessary existence condition in the setting of Theorem \ref{t-ordinary1} is
$$\sum_{i=2}^m\frac{i(i+1)}{2}r_i\le\frac{d(d+3)}{2},$$
as long as $d$ is big enough, which follows from the Alexander-Hirschowitz theorem \cite[Theorem 1.1]{AH1} (see also \cite[Theorem 3.4.22]{GLS3}).
That is, the bound (\ref{eq:ordinary2}) is asymptotically optimal (even for non-T-smooth families).

\begin{proof}
The first part of the theorem follows, in fact, from the Alexander-Hirschowitz theorem \cite[Theorem 1.1]{AH1} after some routine work ensuring that a generic bivariate polynomial of degree $d$, whose derivatives vanish up to appropriate order at the given points in general position, defines an irreducible curve with only ordinary singularities as prescribed.
The second part, however, is not accessible within this framework, since the control over the real singularity types may require at least $\frac{d^2}{m}$ extra independent conditions, which is not bounded by a linear function of $d$. So, to prove the second statement (and thereby the first one), we apply a suitable version of the patchworking construction.

The main element of the construction consists of a collection of the following patchworking patterns. Fix any $2\le i\le m$ and consider the lattice rectangle $R_i=\conv\{(0,0),(0,i),(i+1,0),(i+1,i)\}$. We claim that there exists a real irreducible polynomial $F_i(x,y)$ with Newton polygon $R_i$ (see Figure \ref{fig3}(a)) which defines a real plane curve having in $(\CC^*)^2$ exactly two singular points:
\begin{itemize}\item either two complex conjugate ordinary singularities of order $i$,
\item or two real ordinary singularities with the prescribed number $j\le i/2$ of pairs of complex conjugate local branches.\end{itemize}
Indeed, in the projective plane $\PP^2$ with coordinates $x,y,z$ consider the pencil of conics passing through the points $(1,0,0)$ and $(0,1,0)$ and through two more points $p_1,p_2\in(\CC^*)^2=\PP^2\setminus\{xyz=0\}$, either a pair of complex conjugate points, or a pair of real generic points. In the case of  complex conjugate $p_1,p_2$, we pick $i$ distinct smooth real conics $Q_1,...,Q_i$ in our pencil, while in the case of real $p_1,p_2$, we pick $i$ smooth conics $Q_1,...,Q_i$ in our pencil so that $i=2j$ of them are real and the others form $j$ pairs of complex conjugate conics. Consider the projective curve $Q_1\cdots Q_i\cdot L=0$, where $L=x+\lambda z$ with a generic real number $\lambda$. This curve has ordinary singularities of order $i$ at $p_1,p_2$, and $(0,1,0)$, an ordinary singularity of order $i+1$ at $(1,0,0)$, and $i$ more nodes, which are intersection points of the line $L=0$ with the conics. There exists a small real deformation of the considered curve that preserves the ordinary singularities at $p_1$, $p_2$, $(1,0,0)$, and $(0,1,0)$ and smoothes out all the extra nodes. This follows directly from \cite[Theorem, page 31]{Sh4} or, after the blowing up $\Sigma\to\PP^2$ of the four ordinary singularities, from \cite[Theorem 6.1(iii)]{GK}, since each component $C$ of the blown-up curve satisfies $-CK_\Sigma>0$, and the nodes do not contribute to the right-hand side of the required inequalities (see also \cite[Theorem 4.4.1(b) (formula (4.4.1.3)), Proposition 4.4.3(b) (formula (4.4.1.10)), and Remark 4.4.212(ii)]{GLS3}). So, we obtain the desired polynomial $F_i$ by substituting $z=1$ into the equation of the above deformed curve.

\begin{figure}
\setlength{\unitlength}{1.0mm}
\begin{picture}(110,80)(10,0)
\thinlines

\put(10,10){\vector(1,0){35}}\put(10,10){\vector(0,1){30}}
\put(55,10){\vector(1,0){70}}\put(55,10){\vector(0,1){70}}

\dashline{1}(55,45)(85,45)\dashline{1}(90,40)(90,35)\dashline{1}(90,35)(95,35)

\thicklines

\put(10,10){\line(1,0){30}}\put(10,10){\line(0,1){20}}
\put(10,30){\line(1,0){30}}\put(40,10){\line(0,1){20}}
\put(55,10){\line(1,0){65}}\put(55,10){\line(0,1){65}}
\put(55,75){\line(1,-1){65}}\put(55,20){\line(1,0){45}}
\put(55,30){\line(1,0){45}}\put(55,40){\line(1,0){30}}
\put(70,20){\line(0,1){20}}\put(85,20){\line(0,1){20}}
\put(100,20){\line(0,1){10}}

\put(7,29){$i$}\put(37,6){$i+1$}\put(51,19){$s$}\put(45,29){$s+i$}
\put(43,39){$s+ki$}\put(37,44){$s+ki+1$}\put(52,74){$d$}\put(119,6){$d$}
\put(25,3){\rm (a)}\put(88,3){\rm (b)}

\end{picture}
\caption{Patchworking of real curves with ordinary singularities}\label{fig3}
\end{figure}
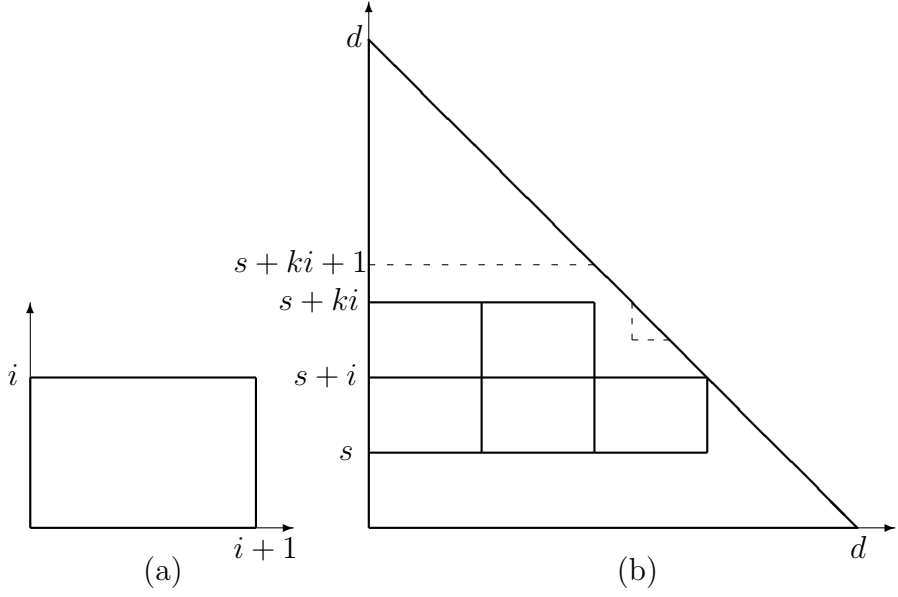

Then we subdivide the lattice triangle $T_d=\conv\{(0,0),(d,0),(0,d)\}$ into convex lattice polygons, in which the desired number of patches corresponding to any fixed real ordinary singularity type should be arranged as shown in Figure \ref{fig3}(b). The polynomials for each rectangle are obtained from just one suitable polynomial $F_i$ constructed above, which should be multiplied by an appropriate monomial and undergo the coordinate change $x\mapsto x^{-1}$ and/or $y\mapsto y^{-1}$ in order to make any two neighboring polynomials agree along the common side of their Newton rectangles. Between the unions of rectangles corresponding to different types of ordinary singularities, we leave the space of vertical size at most $m-1$ (shown by dashed lines in Figure \ref{fig3}(b)) which should be subdivided into lattice triangles with the associated polynomials defining smooth real curves in $(\CC^*)^2$ so that finally one obtains a convex subdivision of $T_d$. To apply the patchworking theorem \cite[Theorem 3.6]{Sh12} we have to verify the topological transversality conditions. With an appropriate orientation of the adjacency graph of the tiles of the subdivision, we get $\partial_+R_i$ in each rectangle $R_i$ to be the union of the bottom and the left sides. The sufficient transversality condition stated in \cite[Theorem 4.1(1), the first formula]{Sh12} (see also \cite[Proposition 4.5.2 and Corollary 1.2.22]{GLS3}) reads as
$$\text{the contribution of the two}\ i-\text{multiple points}\ =\ 2i$$
$$<2i+1\ =\ \text{the total length of the upper and the right sides of}\ R_i.$$

Finally, we note that the resulting curve admits a deformation smoothing out any prescribed singularities and keeping the remaining ones. Thus, one can realize any values in the right-hand side of (\ref{eq:ordinary2}) with $\phi^{ordinary}_{m,fix}(d)\sim m^2d$. \end{proof}

\medskip\noindent
{\bf Curves with semi-quasihomogeneous singularities.}
Consider curve singularities topologically equivalent to $x^m+y^n=0$, $2\le m\le n$. A slightly modified Zariski construction mentioned in Section \ref{ssec:2.2} provides a series of curves with a large number of semi-quasihomogeneous singularities. Combining Zariski's with the patchworking construction we obtain even a series of  curves belonging to a $T$-smooth family.

\begin{theorem}\label{t-sqh}\index{existence!of q.h. singularities}
Let $2\le m\le n$.

(1) If $\gcd(m,n)=1$, then for any $r\ge1$, the curve $C$ of degree $d=rmn$ given by $F^n+G^m=0$, where $F,G\in\CC[x,y,z]$ are generic homogeneous polynomials of degree $rm$ and $rn$, respectively, is irreducible and has $r^2mn=\frac{d^2}{mn}$ singular points of type $x^m+y^n=0$ as its only singularities.

(2) If $\gcd(m,n)>1$, then for any $r\ge1$, the curve $C$ of degree $d=rmn+1$ given by $L_1F^n+L_2G^m=0$, where $F,G\in \CC[x,y,z]$ are generic homogeneous polynomials of degree $rm$ and $rn$, respectively and $L_1,L_2$ are generic linear polynomials, is irreducible and has $r^2mn=\frac{(d-1)^2}{mn}$ singular points of type $x^m+y^n=0$ as its only singularities.

Moreover, in both cases we can assume that the constructed curves are real and all their singular points are real.
\end{theorem}

\begin{proof}
By construction, the curves $F=0$ and $G=0$ intersect transversally at $r^2mn$ distinct points, and the curve $C$ has the topological singularity $x^m+y^n=0$ at each of these intersection points.
\end{proof}

Note that the number of singularities obtained on these  ``Zariski curves'' is close to the genus bound (\ref{genbound}): for instance, under the conditions of Theorem \ref{t-sqh}, the number of the considered singular points does not exceed
$$\frac{(d-1)(d-2)}{(m-1)(n-1)+\gcd(m,n)-1}<\frac{(d-1)^2}{(m-1)(n-1)},$$ which is comparable with the actual numbers $\frac{d^2}{mn}$ and $\frac{(d-1)^2}{mn}$ of singularities.
On the other hand, we cannot guarantee that the conditions imposed on the curve by singular points are independent; hence, we cannot ensure that there exist curves of the given degree with any intermediate amount of singular points of the given type.
\medskip

One can modify Zariski's construction further and obtain a curve with different semi-quasihomogeneous singularities, and, moreover, obtain a curve belonging to a $T$-smooth equisingular stratum:

\begin{theorem}\label{t-squ1} \index{existence!of q.h. singularities}
(1) Given integers $$d_0,d'_0\ge0,\ d_1,...,d_k,d'_1,...,d'_l\ge1, \ m_1,...,m_k,n_1,...,n_l\ge2$$ such that
$$d=\sum_{i=1}^km_id_i+d_0=\sum_{j=1}^ln_jd'_j+d'_0.$$
Then   the plane curve of degree $d$
\begin{equation}H\prod_{i=1}^kF_i^{m_i}+H'\prod_{j=1}^lG_J^{n_j}=0,\label{eq:curve}\end{equation} where $F_1,...,F_k,G_1,...,G_l,H,H'$ are generic homogeneous polynomials of degree $d_1,...,d_k,d'_1,...,d'_l,d_0,d'_0$, respectively, has $d_id'_j$ singular points of topological type $x^{m_i}+y^{n_j}=0$ for all $i=1,...,k$, $j=1,...,l$. Further on,
we can achieve that the constructed curve is real and all its singular points are real.

(2) In addition, if  either
\begin{enumerate}\item[(i)]
$$\begin{cases}&\Big[d_0\ne d'_0\quad\text{or}\quad\gcd(m_1,...,m_k,n_1,...,n_l)=1\Big]\\
&\text{and}\quad \sum_{i=1}^k\sum_{j=1}^ld_id'_j(m_i+n_j-\gcd(m_i,n_j)-\eps(m_i,n_j))<\widetilde d(d'+3),\end{cases}$$ where $d'=\min\{d_0,d'_0\}$, $\widetilde d=d-d'$,
$$\eps(a,b)=\begin{cases}1,\ &a\equiv 1\mod b,\ \text{or}\ b\equiv 1\mod a,\\ 0,\ &\text{otherwise},\end{cases}$$
or
\item[(ii)]
\begin{equation}\sum_{i=1}^k\sum_{j=1}^ld_id'_j(m_i+n_j-\gcd(m_i,n_j)-\eps(m_i,n_j))<\widetilde d(d'+2),
\label{eq:sqhtsmooth}\end{equation}
\end{enumerate}
then the curve constructed above belongs to a $T$-smooth equisingular family, and it admits a deformation moving all its singular points to a general position.
\medskip

\end{theorem}

\begin{proof}
We have to explain only part (2). Consider the case (2ii). We can assume that $d_0=d'_0=d'>0$. We shall prove that the curve given by
\begin{equation}\widehat H\left(L\prod_{i=1}^kF_i^{m_i}+L'\prod_{j=1}^lG_J^{n_j}\right)=0,\label{eq:curve1}\end{equation} where $F_1,...,F_k,G_1,...,G_l,L,L',\widehat H$ are generic homogeneous polynomials of degree $d_1,...,d_k,d'_1,...,d'_l,1,1,d'-1$, respectively, belongs to a $T$-smooth equisingular family with respect to the singular points located at the set $\bigcup_{i,j}\{F_i=0\}\cap\{G_j=0\}$, and all these singular points can be moved to a general position (while the nodes located on $H$ may disappear). Since these properties are open, they will hold for the original curve (\ref{eq:curve}) as well. On the other hand, by \cite[Theorem in page 31]{Sh4} the required properties will follow from \cite[Inequality (4)]{Sh4} which in our situation takes the form of (\ref{eq:sqhtsmooth}) (cf. the definition of the invariant $b$ in \cite[Definition 2]{Sh4}). In the same manner, one can settle the case (2i).
\end{proof}

\begin{remark}{\em (1) Under the hypotheses of part (2) of Theorem \ref{t-squ1}, one can deform the curve (\ref{eq:curve}) smoothing out prescribed singularities and keeping the other ones.

(2) The hypotheses of Theorem \ref{t-squ1}(2) can be relaxed to the following one:
$$d'\ge\widetilde d\cdot\max_{i,j}\left\{\frac{1}{m_i}+\frac{1}{n_j}\right\}-2\ .$$}
\end{remark}

\subsection{Curves with arbitrary singularities}\label{ssec:3.3}

%
%
%

Note that none of the constructions discussed in Sections \ref{sec:2} and \ref{ssec:3.2} can be applied directly, if we ask about the existence of plane curves of a given degree with a prescribed collection of topological or analytic singularities, which are not further specified. For instance, the patchworking construction requires to find a patchworking pattern of possibly minimal size which, for an arbitrary singularity type, is not clear at all.

However, there is another approach that combines some features of  the patchworking construction\index{construction!patchworking}\index{patchworking!construction} with suitable $H^1$-vanishing criteria for the ideal sheaves of zero-dimensional subschemes of the plane placed in general position. The first attempt like that, undertaken in \cite{GLS1}, has led to the following existence criterion: for any positive integer $d$ and topological singularity types $S_1,...,S_r$, the inequality
\begin{equation}\sum_{i=1}^r\mu(S_i)\le\frac{d^2}{392}\label{eq:top}\end{equation} is sufficient for the existence of an irreducible plane curve of degree $d$ having $r$ singular points of types $S_1,...,S_r$, respectively, as its only singularities. This criterion already possessed the two important properties: it was \emph{universal}, i.e., uniformly applicable to arbitrary topological singularities, and \emph{asymptotically proper} i.e. comparable with the necessary condition (\ref{mubound}). On the other hand, the analytic singularity types were left aside, and the coefficient of $d^2$ in (\ref{eq:top}) was too small. An improved method was suggested in \cite{Sh13} (for a further improvement and a detailed exposition see \cite[Section 4.5.5]{GLS3}). The main result was (cf. \cite[Theorem 3 and Remark 5]{Sh13} and
\cite[Corollary 4.5.15]{GLS3}):

\begin{theorem}\label{t-analytic}\index{existence!of arbitrary singularities}
For any positive integer $d$ and an arbitrary sequence of complex, resp. real, analytic singularity types $S_1,...,S_r$, the inequality
\begin{equation}\sum_{i=1}^r\mu(S_i)\le\frac{1}{9}(d^2-2d+3)\label{eq:analytic}
\end{equation}
is sufficient for the existence of a plane irreducible complex, resp. real, curve of degree $d$ having $r$ singular points of types $S_1,...,S_r$, respectively, as its only singularities. Moreover, the position of the singular points together with the tangent directions for unibranch singularities can be chosen generically.
\end{theorem}

We point out that condition (\ref{eq:analytic}) has a much larger coefficient of $d^2$ in the right-hand side than (\ref{eq:top}) and covers both arbitrary topological and analytic singularity types, being universal with respect to the choice of singularities.

We also remark that condition (\ref{eq:analytic}) is a weaker form of the following stronger sufficient existence conditions (see \cite[Theorem 4.5.14]{GLS3}):
$$6n+10k+\frac{49}{6}t+\frac{625}{48}\sum_{S_i\ne A_1,A_2}\delta(S_i)\le d^2-2d+3,$$
for topological singularity types $S_1,...,S_r$, and
$$6n+10k+\sum_{S_i\ne A_1,A_2}\frac{7\mu(S_i)+2\delta(S_i))^2}{6\mu(S_i)+3\delta(S_i)}\le d^2-2d+3,$$
for  analytic singularity types $S_1,...,S_r$. In these inequalities, $n$ is the number of nodes $A_1$, $k$ is the number of cusps $A_2$, and $t$ is the number of singularities $A_{2m}$, $m\ge2$, occurring in the list $S_1,...,S_r$.
Notice that the coefficients in front of $n$ and $k$ in the above two formulas are the best possible, since a node at a prescribed point imposes $3$ conditions, while a cusp at a fixed point with a fixed tangent direction imposes $5$ conditions.

The case of just one singularity is of special interest. The corresponding result sounds as follows (\cite[Theorem 2 and Remark 5]{Sh13} and \cite[Theorem 4.5.19]{GLS3}):

\begin{theorem}\label{t-onesingularity}\index{existence!of arbitrary singularities}
For an arbitrary analytic singularity type $S$, there exists a plane curve of degree
\begin{equation}d\le3\sqrt{\mu(S)}-1\label{eq:onesingularity}\end{equation} having a singular point of type $S$ as its only singularity.
\end{theorem}

\begin{remark}\label{r-arb1}{\em
A necessary condition for $d$ as in Theorem \ref{t-onesingularity} comes, for instance, from (\ref{mubound}): $d\ge\sqrt{\mu(S)}+1$. Thus, the sufficient condition (\ref{eq:onesingularity}) is of the same order $\sqrt{\mu(S)}$ as the necessary one.}
\end{remark}

We comment on the main ideas behind Theorems \ref{t-analytic} and \ref{t-onesingularity}:

The first important ingredient is to reduce the existence problem to an $H^1$-vanishing condition for the ideal sheaf of a suitable zero-dimensional subscheme of the plane. Namely, to each reduced plane curve germ $(C,p)$ we associate two zero-dimensional schemes $Z^{s}_{st}(C,p)$ and $\widetilde Z^{a}_{st}(C,p)$ in $\PP^2$ supported at $p$ that are defined  as follows (cf. \cite[Section 1.1.4 and 4.5.5.1]{GLS3}):
\begin{itemize}\item Take the complete resolution tree ${\mathcal T}^\infty(C,p)$, choose the subtree ${\mathcal T}^*(C,p)$ containing all infinitely near points which are not the nodes of the union of the strict transform of $(C,p)$ with the exceptional locus, and then define $Z^s_{st}(C,z)$ by the ideal $I^s_{st}\subset{\mathcal O}_{\PP^2,p}$ generated by the elements $\varphi\in{\mathcal O}_{\PP^2,p}$ having the multiplicity $\mt(C,p)+1$ at $p$, and the multiplicity of the strict transform of $(C,p)$ at each infinitely near point $q\in{\mathcal T}^*(C,p)\setminus\{p\}$.
\item Let $(C,p)$ be given by $f(x,y)=0$ with $f\in{\mathcal O}_{\PP^2,p}$ square-free, $x,y$ affine coordinates in a neighborhood of $p$ such that $p=(0,0)$. Define $\widetilde Z^a_{st}(C,z)$ by the ideal
${\mathfrak m}_p\widetilde I^a\subset{\mathcal O}_{\PP^2,p}$, where
$$\widetilde I^a=\{g\in{\mathcal O}_{\PP^2,p}\ :\ g,g_x,g_y\in\langle f,f_x,f_y\rangle\}.$$
\end{itemize}
The importance of these schemes comes from the following claim (cf. \cite[Proposition 4.5.12]{GLS3}):

\begin{lemma}\label{l-arb1}
Let $Z$ denote $Z^s_{st}(C,p)$, resp. $\widetilde Z^a_{st}(C,p)$. If a positive integer $d$ satisfies
$$H^1(\PP^2,{\mathcal J}_{Z/\PP^2}(d))=0,$$
where ${\mathcal J}_{Z/\PP^2}$ is the ideal sheaf of the subscheme $Z\subset\PP^2$, then there exists a curve $C'\subset\PP^2$ of degree $d$ such that the germ $(C',p)$ is topologically, resp. analytic equivalent to $(C,p)$.
Furthermore, the corresponding equisingular family $V^{irr}_d(S)$ ($S$ being the topological, resp. analytic type of $(C,z)$) is T-smooth at $C'$.
\end{lemma}

The length of $Z^s_{st}(C,z)$ and $\widetilde Z^a_{st}(C,z)$ can be estimated from above by a linear function in $\delta(C,p)$ and $\mu(C,p)$ (see \cite[Corollary 1.1.4 and Lemma 1.1.78]{GLS3}).
\medskip

It is not difficult to see that, for a randomly chosen germ $(C,p)$, the minimal $d$ in Lemma \ref{l-arb1} can be of order $\deg Z$, i.e., of order $\mu(C,p)$, but $\sqrt{\mu(C,p)}$ is required in Theorems \ref{t-analytic} and \ref{t-onesingularity}. So, the second important idea is to replace the germ $(C,p)$, or, more precisely, the corresponding zero-dimensional scheme $Z$ by a generic element in $\Iso(Z)$, the orbit of $Z$ by the action of the group
$\Aut({\mathcal O}_{\PP^2,p})$. The principal bound is as follows (cf. \cite[Propositions 8 and 10, Remark 3]{Sh13} and \cite[Proposition 3.6.1 and Corollary 3.6.4]{GLS3}). Given an irreducible zero-dimensional scheme $Z\subset\PP^2$ supported at $z\in\PP^2$, denote by $M_2(Z)$ the intersection multiplicity of two generic elements of the ideal $I(Z)\subset{\mathcal O}_{\PP^2,z}$. If $Z$ consists of irreducible components $Z_1,...,Z_k$, we set $M_2(Z)=M_2(Z_1)+...+M_2(Z_k)$.

\begin{lemma}\label{l-arb2}
For an arbitrary zero-dimensional $Z$ of degree $\deg Z>2$, there exists $Z'\in\Iso(Z)$ and
$$d\le\frac{\deg Z}{\sqrt{\frac{4}{3}M_2(Z)}}+\sqrt{\frac{4}{3}M_2(Z)}-2$$ such that
$$H^1(\PP^2,{\mathcal J}_{Z'/\PP^2}(d))=0\ .$$

In the case of a reducible scheme $Z=Z_1\cup...\cup Z_r$, we move its supporting points to a general position and choose generic element $Z'_i\in\Iso(Z_i)$ for each component $Z_i$.
\end{lemma}

A proper combination of these two ideas (Lemma \ref{l-arb1} and \ref{l-arb2}) leads to Theorems \ref{t-analytic} and \ref{t-onesingularity}.\\

The following diagram illustrates the present knowledge about the existence of
plane curves of degree $d$ with arbitrary (analytical or topological) singularities $S_1,...,S_r$.

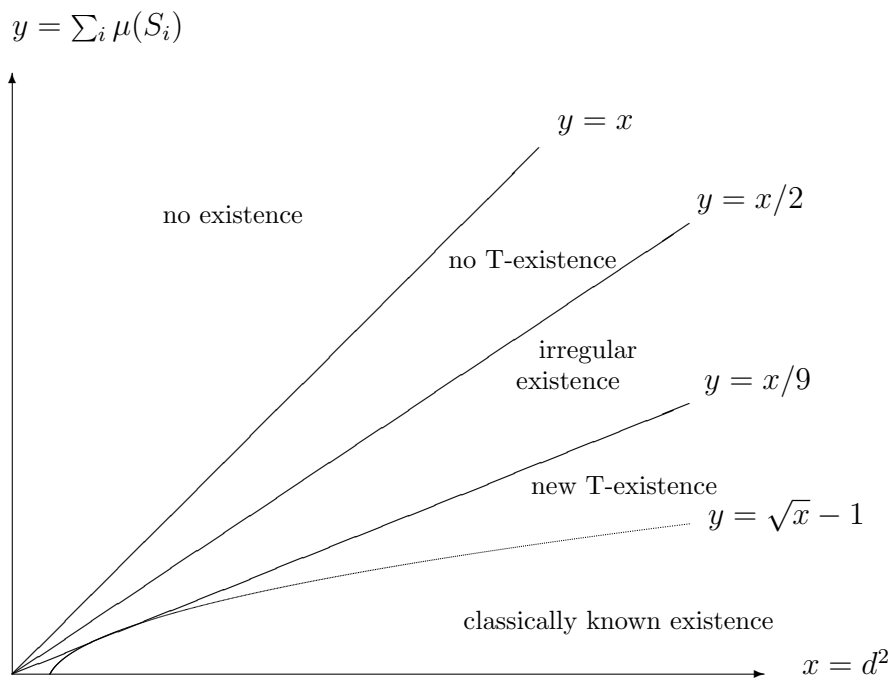
\begin{figure}[ht]
\setlength{\unitlength}{1cm}
\begin{picture}(11,10)(3,0)
\thinlines \put(3,1){\vector(1,0){10}} \put(3,1){\vector(0,1){8}}
\put(3,9.5){$y=\sum_i\mu(S_i)$} \put(13.5,1){$x=d^2$}
\put(3,1){\line(1,1){7}} \put(3,1){\line(5,2){9}}
\put(3,1){\line(3,2){9}}
\qbezier(3.5,1)(4,2)(12,3) \put(10.25,8.25){$y=x$}
\put(12.2,4.8){$y=x/9$}\put(12.1,7.2){$y=x/2$}
\put(12.25,3){$y=\sqrt{x}-1$} \put(5,7){\rm {\footnotesize no
existence}}\put(10,5.2){\rm  {\footnotesize irregular}}\put(9.7,4.8){\rm  {\footnotesize
existence}}\put(9.9,3.4){\rm  {\footnotesize new T-existence}}\put(9,1.6){\rm
 {\footnotesize classically known existence}}\put(8.8,6.4){\rm {\footnotesize  {\footnotesize no T-existence}}}
\end{picture}
\caption{Necessary and sufficient conditions for existence: The
statement holds asymptotically for the region between the
corresponding boundary curves}
\end{figure}

\section{Related and open problems}\label{sec:related}

\subsection{Existence vs. $T$-smoothness and irreducibility}\label{ssec:smoothirred}

The existence problem for singular algebraic curves is tightly related to the geometry of the corresponding equisingular family, especially, to the $T$-smoothness property, which was crucial in the patchworking construction and also used in other constructions. In this section, we construct singular plane algebraic curves of two sorts: \\
(i) those which demonstrate the sharpness of the known $T$-smoothness criteria and \\
(ii) those which yields examples of reducible equisingular families.
\medskip

Recall the {\em universal sufficient conditions for the $T$-smoothness} (see \cite[Theorem 1]{Sh2}, \cite[Corollary 3.9(d)]{GL}, \cite[Theorems 1 and 2]{GLS5} and \cite[Theorems 4.3.8 and 4.3.9]
{GLS3}):

\begin{theorem}\label{t-tsmoothirred}\index{existence!T-smooth}
Let $C$ be an irreducible plane curve of degree $d$ with singular points $p_1,...,p_r$ of topological or  analytic types $S_1,...,S_r$ respectively. Then the equisingular family $V^{irr}_d(S_1,...,S_r)$ is $T$-smooth at $C$ if
either
\begin{equation}\sum_{i=1}^r\tau'(S_i)<4d-4,\label{eq:4d-4}\end{equation}
where $\tau'=\tau^{es}$ if $S_i$ is a topological type and  $\tau'=\tau$ if $S_i$ is an  analytic  type, $i=1,...,r$, or
\begin{equation}\sum_{i=1}^r\gamma'(C,p_i)\le(d+3)^2\ ,\label{eq:tsmooth}\end{equation}
where $\gamma'=\gamma^{es}$ for  $S_i$ is a topological  type and $\gamma'=\gamma^{ea}$ for $S_i$ an analytic type, $i=1,...,r$.
\end{theorem}

The symbols $\tau^{es},\tau,\gamma^{es},\gamma^{ea}$ 
are topological or  analytic singularity invariants. For precise definitions we refer to \cite[Section 1.2.3.1 and 
Definition 1.1.63]
{GLS3}, and for the detailed study of their properties to \cite[Corollary 1.1.64, 
Proposition 1.2.26]{GLS3}. Here we provide only the following information used below (see also Preliminaries):
\begin{itemize}\item $\tau$ is the Tjurina number, i.e., the dimension of the Tjurina algebra
$$\tau(C,p)=\tau(f)=\dim_\CC{\mathcal O}_{\PP^2,z}/\langle f, f_x, f_y\rangle,$$ where a square-free element $f\in{\mathcal O}_{\PP^2,z}$ defines the curve germ $(C,p)$, and $x,y$ are local coordinates;
\item $\tau^{es}(C,p)$ is the codimension of the $\mu$-const stratum in the versal deformation base of the curve germ $(C,p)$; it always satisfies $$\tau^{es}(S)\le\tau(S)\le\mu(S),$$ with equalities for a simple singularity type $S$.
    \item $\gamma^{es}(C,p)\le\gamma^{ea}(C,p)$ for any curve germ with equality if the singularity is simple, furthermore
        $$\gamma^{ea}(S)\begin{cases}=(\mu(S)+1)^2,\ &\text{if}\ S=A_n,\ n\ge1,\\
        \le(\mu(S)+1)^2,\ &\text{otherwise}.\end{cases}$$
\end{itemize}

We also recall important particular cases (see \cite[Corollaries 4.3.6, 4.3.11, 
formula (1.2.3.1)]{GLS3} and Theorem \ref{t-tsmoothirred}):

\begin{theorem}\label{t-particular}\index{existence!T-smooth}
(1) An irreducible plane curve of degree $d$ with $n$ nodes and $k$ cusps as its only singularities belongs to a $T$-smooth equisingular family if either
\begin{equation}k<3d\ ,\label{eq:3d}\end{equation}
(any $n\geq0$, i.e., nodes do not count) or
\begin{equation}4n+9k\le(d+3)^3\ .\label{eq:nodecusp}\end{equation}

(2) An irreducible plane curve of degree $d$ with $r$ singular points of simple singularity types $S_1,...,S_r$ belongs to a $T$-smooth equisingular family if
\begin{equation}\sum_{S_i\in A}(\mu(S_i)+1)^2+\sum_{S_i\in D,E}\max\left\{(\mu(S_i)-1)^2,\frac{1}{2}(\mu(S_i)+2)^2\right\}\le(d+3)^3.\label{eq:AD}\end{equation}
\end{theorem}

\begin{remark}{\em
Observe that, first, the right-hand sides in the {\em $4d$-criterion} (\ref{eq:4d-4}) and the  {\em $3d$-criterion} (\ref{eq:3d}) for $T$-smoothness are only linear in $d$, while the sufficient condition (\ref{eq:analytic}) for existence is quadratic in $d$ in the right-hand side. Second, the $T$-smoothness restrictions (\ref{eq:3d}) and (\ref{eq:nodecusp}) are far away from the existence conditions in Theorem \ref{t-nodecusp}, and, third, the invariants assigned to singular points in the left-hand side of (\ref{eq:tsmooth}) and 
\ref{eq:AD}) 
are in general of order $\mu^2(S)$, while in (\ref{eq:analytic}) we have just $\mu(S)$. That is, for general analytic or topological types  there exists a wide range of nonempty equisingular families, which do not fall to the limits of Theorems \ref{t-tsmoothirred} and \ref{t-particular}. On the other hand, for curves with simple (Theorem \ref{t-simple1}) , ordinary (Theorem \ref{t-ordinary1}), and semi-quasihomogeneous singularites (Theorem \ref{t-squ1}), we have $T$-smooth equisingular families with asymptotically proper existence bounds.}
\end{remark}

Thus, natural questions arise for arbitrary singularities:

\smallskip{\it (1) Does the difference between the left-hand sides of the sufficient $T$-smoothness
criteria in Theorems \ref{t-tsmoothirred} and \ref{t-particular} and the left-hand side of the sufficient existence criterion (\ref{eq:analytic}) reflect the lack of $T$-smoothness
outside the limits pointed in Theorems \ref{t-tsmoothirred} and \ref{t-particular}?
\smallskip

 (2) Or, can the singularity invariants in the left-hand side of the sufficient $T$-smoothness criteria be essentially diminished?}

\medskip
We exhibit a series of singular plane algebraic curves and their families demonstrating that the linear bounds (\ref{eq:4d-4}) and (\ref{eq:3d}) are sharp, the coefficients in the left-hand side of (\ref{eq:nodecusp}) are sharp, and the singularity invariants in the other $T$-smoothness 
conditions can, in principle, be improved only by a constant factor, while their order with respect to the Milnor number persists. We mainly use the constructions discussed in Sections \ref{sec:2} and \ref{sec:3}. For all details and more examples see \cite[Sections 4.2.3, 4.3.3]
{GLS3} and references therein.

The following example, which is due to du Plessis and Wall \cite{DPW} (elaborated further in \cite{Gou}), shows that the bound (\ref{eq:4d-4}) is sharp.

\begin{theorem}\label{t-DPW}\index{examples!non T-smooth}
For any $d\ge5$ the irreducible curve $C\subset\PP^2$ given by $x_1^d+x_2^5x_0^{d-5}+x_2^d=0$ has the unique singular point $z=(1,0,0)$ with Tjurina number $4d-4$, and the equisingular family $V^{irr}_d(S)$, where $S$ is the analytic type of the germ $(C,z)$, is not $T$-smooth. Furthermore, the family $V^{irr}_d(S)$ is nonreduced for $d\le6$,
consists of two intersecting components for $d=7$, and is reduced, irreducible with a singular locus containing $C$ for $d\ge8$.
\end{theorem}

The examples found in \cite[Proposition 4.5]{GLS5} (see also \cite[Theorem 4.3.23]{GLS3}) show that the classical Severi-Segre-Zariski bound (\ref{eq:3d}) is sharp and that the coefficients $4$ and $9$ in (\ref{eq:nodecusp}) are sharp.

\begin{theorem}\label{t-nodecusp1}\index{eexamples!non T-smooth}
Let $p\ge6$, $q\ge9$. Then the variety $V^{irr}_d(n\cdot A_1,k\cdot A_2)$ has a non-$T$-smooth component if
\begin{enumerate}\item[(a)] $d=3p$, $n=0$, $k=p^2+3p$, or
\item[(b)] $d=2q$, $n=q^2-9q$, $k=6q$.
\end{enumerate}
\end{theorem}

In the series (a), $9k=d^2+9d$, and hence the coefficient $9$ in (\ref{eq:nodecusp}) is sharp. In the series (b), $4n=d^2-18d$ and also $k=3d$, which yields the sharpness of (\ref{eq:3d}) and of the coefficient $4$ in (\ref{eq:nodecusp}). A curve $C\in V^{irr}_d(n\cdot A_1,k\cdot A_2)$ at which the $T$-smoothness fails can be constructed by Zariski's method as in Theorem \ref{t-sqh}: namely, we set $C$ to be given by
$$AP^2R^2+BQ^2=0,$$ where $A,B,P,Q,R$ are generic polynomials of degrees $a,b,p,q,r$, respectively, such that
$$a=d-3p-2r\ge0,\quad b=d-2q\ge0.$$
The series (a) corresponds to $r=0$, $d=3p$, $q=p+3$, $p\ge6$, while the series (b) corresponds to $d=2q$, $p=6$, $r=q-9$. For the failure of the $T$-smoothness see the proof of \cite[Theorem 4.3.23]{GLS3}.

The next series of examples were found in \cite[Theorems 5 and 6]{GLS0} (see also \cite[Theorems 4.3.24 and 4.3.25]{GLS3}). They show that the coefficients of the $A_n$ and $D_n$ singularities in (\ref{eq:AD}) may, in principle, be reduced, but only by a constant factor $\ge\frac{1}{4}$.

\begin{theorem}\label{t-AD1} \index{examples!non T-smooth}
Let $l\ge2$, $0\le s\le l-2$, $q\ge\frac{3}{l-s-1}$ be integers.

(1) Let $k>2l-s$ be an integer. Then there exists an irreducible plane curve $C$ of degree $d=q(k+s)$ having precisely $q^2$ singular points, all of type $A_{kl+s-1}$, such that the family $V^{irr}_d(q^2\cdot A_{kl+s-1})$ is not $T$-smooth at $C$. Moreover,
\begin{itemize}\item if $k>4l-s$, then $C$ belongs to a component of $V^{irr}_d(q^2\cdot A_{kl+s-1})$ of expected dimension which is singular at $C$;
\item if $k\ge\max\{l^2+2l,4l+4-s\}$, then the germ of $V^{irr}_d(q^2\cdot A_{kl+s-1})$ at $C$ is a singular, normal complete intersection.
\end{itemize}

(2) Let $k>2l-s+1$ be integer. Then there exists an irreducible plane curve $C$ of degree $d=q(k+s)$ having precisely $q^2$ singular points, all of type $D_{kl+s+1}$, such that the family $V^{irr}_d(q^2\cdot D_{kl+s+1})$ is not $T$-smooth at $C$. Moreover,
\begin{itemize}\item if $k>4l-s$, then $C$ belongs to a component of $V^{irr}_d(q^2\cdot D_{kl+s+1})$ of expected dimension which is singular at $C$;
\item if $k\ge\max\{l^2+2l+2,4l+4-s\}$, then $V^{irr}_d(q^2\cdot D_{kl+s+1})$ is a singular, normal locally complete intersection at $C$.
\end{itemize}
\end{theorem}

In particular, for $l=2$, $s=0$, we have in part (1)
$$d^2=q^2k^2\quad\text{and}\quad q^2(\mu(A_{2k-1})+1)^2=4q^2k^2,$$
and in part (2)
$$d^2=q^2k^2\quad\text{and}\quad q^2(\mu(D_{2k+1})-1)^2=4q^2k^2.$$
In part (1) the construction is as follows. Take the affine curve
$$(y+y^l-x^l)^2(1+\lambda_1x^{k+s-2l}+\lambda_2x^sy^k+y^{k+s})=0,$$
where $\lambda_1,\lambda_2\in\CC$ are generic. It is easy to check that it is irreducible with the unique singularity $A_{kl+s-1}$ at the origin. Then we take its projective closure, choose a generic projective coordinate system $(x_0,x_1,x_2)$ and apply the transformation $(x_0,x_1,x_2)=(y_0^q,y_1^q,y_2^q)$ (cf. Ivinskis' and Hirano's constructions \cite{Iv,HA}). In part (2), we start with the affine curve
$$x(y+y^l-x_l)^2(1+\lambda_1x^{k+s_2l-1})+\lambda_2x^sy^k+y^{k+s}=0,$$ which is irreducible and has the unique singularity $D_{kl+s+1}$ at the origin. Then similarly take the projective closure and apply the transformation $(x_0,x_1,x_2)=(y_0^q,y_1^q,y_2^q)$ in generic projective coordinates $x_0,x_1,x_2$. For the lack of the $T$-smoothness, we refer to \cite[Section 4.3.3.2]{GLS3}.

\begin{remark}\label{r-nontsmooth}{\em
The first ever (finitely many) examples of reduced equisingular families of expected dimension which are not smooth are due to Luengo \cite{Lu,Lu1}. In particular, one of his examples comes with a curve of degree $9$ with a unique singularity $A_{35}$.
We recover this curve in Theorem \ref{t-AD1}(1) for $l=4$, $s=0$, $q=1$, $k=9$.}
\end{remark}

Another application of the construction methods discussed in Sections \ref{sec:2} and \ref{sec:3} is to find interesting examples of reducible equisingular families. There are several ways to verify that an equisingular family is reducible:
\begin{itemize}\item An explicit computation of the equisingular family. This is available only for very specific situations for relatively small degrees, see, for example, \cite[Theorem 1.1(ii)]{Gou}, where the family $V^{irr}_7(S)$ ($S$ being the analytic type of the singularity $x^5+y^7=0$) was shown to be reducible.
\item Exhibiting two (or more) components of an equisingular family, whose generic members differ from the algebraic-geometric point of view. In the classical example by Zariski \cite[Sections VIII.3 and VIII.5]{Za2}, the family $V^{irr}_6(6\cdot A_2)$ contains (at least) two components: in one of them the generic curve has $6$ cusps on a conic, while on the other one this is not the case.
    \item Exhibiting two (or more) components of the equisingular family, whose generic members are embedded into the plane in a topologically different way. The aforementioned Zariski example is of such kind \cite{Za0}, since the complements to these generic curves in the plane are not homeomorphic (they have different Alexander polynomials and different fundamental groups).
        \item Exhibiting two (or more) components of the equisingular family that have different dimensions, or such that one is reduced (for instance, $T$-smooth) and the other is not.
\end{itemize}
We present here examples of the last kind. In fact, equisingular families with components of dimension higher than the expected one were known for a while: Segre \cite{Segre2} (see also \cite{Ta0}) showed that the dimension of the component of $V^{irr}_{6m}(6m^2\cdot A_2)$ containing the curves $F^3+G^2=0$, $\deg F=2m$, $\deg G=3m$, exceeds the expected dimension by at least
$$\frac{(m-1)(m-2)}{2}>0\quad\text{as long as}\quad m\ge3,$$
Wahl \cite{Wa1} showed that the family $V^{irr}_{104}(3636\cdot A_1,9000\cdot A_2)$ contains a nonreduced component.
The problem is to show that there exists another (say, $T$-smooth) component of the considered equisingular family.

According to \cite[Sections VIII.3 and VIII.5]{Za2}, \cite[Theorem 2.1]{Sh7}, \cite[Proposition 5.4]{GLS5}, \cite[Proposition 1.1]{GLS4} (see also \cite[Examples 4.2.0.9 and 4.2.0.10, Propositions 4.6.10 and 4.6.11]{GLS3}), we have the following statement.

\begin{theorem} \index{examples!non T-smooth}\label{t-red}

(1) Each of the families $V^{irr}_6(6\cdot A_2)$ and $V_{12}(24\cdot A_2)$ has (at least) two distinct components of the expected dimension.

(2) Let $p,d$ be integers satisfying
$$p\ge3,\quad 6p\le d\le12p-\frac{3}{2}-\sqrt{35p^2-15p+\frac{1}{4}}\ .$$
Then the family $V^{irr}_d(6p^2\cdot A_2)$ has components of different dimensions. Moreover, if $d>6p$, then $\pi_1(\PP^2\setminus C)\simeq\Z/d\Z$ for all curves $C\in V^{irr}_d(6p^2\cdot A_2)$.

(3) Let $m\ge9$ Then there exists $k_0=k_0(m)$ such that for any $k\ge k_0$ and any integer $s$ atisfying
$$\frac{k-1}{2}\le s\le k\left(1-\sqrt{\frac{2}{m}}\right)-\frac{3}{2}\ ,$$ the equisingular family
$V^{irr}_{km+s}(k^2\cdot S(m))$ of irreducible plane curves of degree $d=km+s$ with $k^2$ ordinary $m$-fold points (topological type $S(m)$) has components of different dimensions. Moreover, \mbox{$\pi_1(\PP^2\setminus C)\simeq\Z/d\Z$} for all curves \mbox{$C\in V^{irr}_d(k^2\cdot S(m))$}.
\end{theorem}

In part (1) the former family is the classical Zariski's example discussed above. The latter family $V^{irr}_{12}(24\cdot A_2)$ contains a component of expected dimension $42$ formed by the curves given by
$$F^3+G^2=0,\quad\deg F=4,\quad\deg G=6,$$ whose $24$ cusps lie on a plane quartic curve. However, according to Theorem \ref{t-nodecusp}(1) there exists a $T$-smooth component of the family $V^{irr}_{12}(28\cdot A_2)$. Smoothing out any four cusps of a curve $C\in V^{irr}_{12}(28\cdot A_2)$, one obtains curves in $T$-smooth components of $V^{irr}_{12}(24\cdot A_2)$, and it is not difficult to verify that there is a $24$-tuple of cusps of $C\in V^{irr}_{12}(28\cdot A_2)$ which does not lie on any plane quartic curve.

In parts (2), resp. (3), one obtains components of the equisingular families of dimension above the expected one formed by the Zariski type curves (cf. Theorem \ref{t-squ1}) given by
$C_{2p}^2C'_{d-6p}+C_{3p}^2C''_{d-6p}=0$, where $$\deg C_{2p}=2p,\ \deg C_{3p}=3p,\ \deg C'_{d-6p}=\deg C''_{d-6p}=d-6p,$$
resp.
$\sum_{i=0}^mR_s^{(i)}F_k^iG_k^{m-k}=0$, where $$\deg F_k=\deg G_k=k,\ \deg R_s^{(i)}=s,\ i=0,...,m.$$
However, in both the cases there exists a $T$-smooth component (of expected dimension): in part (2) by Theorem \ref{t-nodecusp}(i), in part (3) this can be derived by means of the Alexander-Hirschowitz theorem \cite[Theorem 1.1]{AH1} (see also \cite[Theorem 3.4.22]{GLS3}).
The fact that the fundamental group of the complement to the considered curves is always abelian follows from Nori's theorem \cite[Proposition 3.27]{Nor}.

\subsection{Curves on other algebraic surfaces}

For other algebraic surfaces than $\PP^2$, we consider only the case of nodal curves, which is the most important one, since it is directly related to the vanishing/nonvanishing of Gromov-Witten invariants.

In the following cases we know complete answers in which the equisingular family is $T$-smooth \cite{XiChen,XiChen1,CC,GLS2,Ta2} (see also \cite[Section 4.5.6.3]{GLS3}).

\begin{theorem}\label{t-others} \index{existence!of nodal curves!on surface}
(1) Let $\Sigma$ be a toric surface associated with the planar nondegenerate lattice polygon $\Delta$, ${\mathcal L}(\Delta)$ the corresponding tautological line bundle. Then the inequality
$$0\le n\le\#(\Int(\Delta)\cap\Z^2)$$
is necessary and sufficient for the existence of an irreducible curve with $n$ nodes (as its only singularities) in the linear system $|{\mathcal L}(\Delta)|$.

(2) Let $\Sigma=\PP^2_k$, $1\le k\le 9$, be the plane blown up at $k$ distinct generic points, $D$ an effective divisor class of type $D=dL-d_1E_1-...-d_kE_k$, where $L$ is the lift of a general line on $\PP^2$, $E_1,...,E_k$ are exceptional divisors, $d\ge d_1\ge...\ge d_k>0$. Suppose that $-DK_\Sigma>0$. An irreducible curve $C\in|D|$ with $n$ nodes as its only singularities exists if and only if either
$$k=1,\quad 0\le n\le p_a(D)=\frac{D^2+DK_\Sigma}{2}+1,$$
or
$$k=2,\quad 0\le n\le p_a(D),\quad \begin{cases}&\text{either}\ d\ge d_1+d_2,\\
&\text{or}\ d=d_1=d_2=1,\end{cases}$$
or
$$k\ge3,\quad D^2>0,\quad 0\le n\le p_a(D).$$

(3) For any $g\ge3$, given a general smooth K3 surface $\Sigma$ of the principal series in $\PP^g$, and $m>0$ and $n$ satisfying
$$0\le n\le \dim|{\mathcal O}_\Sigma(m)|,$$ there exists an irreducible curve in the linear system $|{\mathcal O}_\Sigma(m)|$ with $n$ nodes as its only singularities.

(4) Let $\Sigma\subset\PP^3$ be a generic smooth surface of degree $d\ge5$. Then, for all
$$m\ge d,\quad 0\le n\le\dim|{\mathcal O}_\Sigma(m)|$$ there exists an irreducible curve in the linear system $|{\mathcal O}_\Sigma(m)|$ having $n$ nodes as its only singularities.
\end{theorem}

\begin{remark}\label{r-nodalother}{\em
Part (1) is actually well-known, one can find details in \cite[Theorem 4.5.32]{GLS3}.

Part (2) admits an extension to the generic surfaces $\PP^2_k$, $k>10$, with extra restrictions to the divisor $D$ and the number of nodes $n$ (see \cite[Theorem 5]{GLS2} or \cite[Tearem 4.5.30 and Corollary 4.5.31]{GLS3}).

Part (4) is proved by the method resembling the patchworking construction. Namely, the proof goes by induction with the case $d=4$ (settled in part (3)) as the base. The induction step consists in a pair of deformations:
\begin{itemize}\item the union of a generic surface $\Sigma_{d-1}$ of degree $d-1$ with a generic tangent plane $\pi$ to it deforms in a family into a generic smooth surface $\Sigma_d$ of degree $d$;
\item an inscribed deformation of a curve in the central fiber that consists of an irreducible curve in the linear system $|{\mathcal O}_{\Sigma_{d-1}}(m)|$ on $\Sigma_{d-1}$ having $\dim|{\mathcal O}_{\Sigma_{d-1}}(m)|$ nodes, and of a nodal curve in the plane $\pi$ of degree $m$ having $\frac{1}{2}(m-d+2)(m-d+3)$ nodes.
\end{itemize} Under certain transversality conditions, the above central curve can be deformed into an irreducible curve
$C\in|{\mathcal O}_{\Sigma_d}(m)|$ having
$$\dim|{\mathcal O}_{\Sigma_{d-1}}(m)|+\frac{(m-d+2)(m-d+3)}{2}=\dim|{\mathcal O}_{\Sigma_d}(m)|$$ nodes.

It should be noted that Chiantini and Ciliberto \cite[Section 1]{CC} exhibit examples of superabundant nodal curves on surfaces in $\PP^3$: in particular, for $d\ge20$ and $m=3$ there are curves in $|{\mathcal O}_\Sigma(3)|$ with $n>\dim|{\mathcal O}_\Sigma(3)|$ nodes, and for $d\ge8$ and $m\gg0$ there exists a component of the equisingular family of curves with $n<\dim|{\mathcal O}_\Sigma(m)|$ nodes that has a dimension greater than the expected one.}
\end{remark}

One can obtain some sufficient existence conditions for curves with arbitrary singularities on smooth projective surfaces. For a topological or analytic singularity type $S$, denote by $e(S)$ the minimal degree of a reduced plane curve $C$ having a singular point of type $S$ as its only singularity, belonging to a $T$-smooth equisingular family, which intersects transversally with the space of curves passing through the intersection of $C$ with a generic fixed line. Then the following holds (see \cite[Proposition 4.5.26]{GLS3}).

\begin{lemma}\label{l-other}
Let $\Sigma$ be a smooth projective algebraic surface, $D$ an effective divisor on $\Sigma$, $L$ a very ample divisor on $\Sigma$. Given topological or analytic singularity types $S_1,...,S_r$ and a zero-dimensional scheme $Z\subset\Sigma$ defined in some distinct $r$ points $z_1,...,z_r\in\Sigma$ by the powers ${\mathfrak m}_{z_i}^{e(S_i)}\subset{\mathcal O}_{\Sigma,z_i}$ of the maximal ideals so that
$$H^1(\Sigma,{\mathcal J}_{Z/\Sigma}(D-L))=0\quad\text{and}\quad\max_{1\le i\le r}e(S_i)<L(D-L-K_\Sigma)-1,$$
then there exists an irreducible curve $C\in|{\mathcal O}_\Sigma(D)|$ with $r$ singular points of types $S_1,...,S_r$, respectively, as its only singularities.
\end{lemma}

The proof is based on a version of the patchworking construction as it appears in \cite{Sh2000,ST1} (see also \cite[Section 2.3.5]{GLS3}). Some numerical conditions, based on $h^1$-vanishing criteria \cite{Xu}, can be found in \cite[Section 4.5.6.2]{GLS3}.

\subsection{Other related problems}

\noindent{\bf Rational cuspidal curves.} A rational cuspidal curve is a complex rational plane curve homeomorphic to a sphere, equivalently, a  rational plane curve having only irreducible singularities (called (generalized) cusps). They attracted much attention due to their interesting properties and tight links to the Jacobian conjecture, affine algebraic geometry, and birational geometry (see \cite{AM,ZL,
KP1}). The subject definitely deserves a separate full size survey. We only mention one result directly related to the existence problem for singular plane curves \cite[Theorem 1.1]{KP2}:

\begin{theorem}\label{t-kp}
A rational cuspidal curve has at most $4$ singular points.
\end{theorem}

There is a series of classification results for rational cuspidal curves (see references in \cite{KP2}).

\medskip
\noindent {\bf Curves in the higher-dimensional projective spaces.} The general irreducible curve of any degree and genus in $\PP^n$, $n\ge3$, is smooth. Thus, singular spacial curves are of a moderate interest. The question on the number of nodes of an irreducible curve in $\PP^n$, $n\ge3$, of degree $d$ and genus $g$ was studied in \cite{Ta-1,Ta1} over the complex field and in \cite{Pe1} over the real field. For $n\ge3$, the genus of an irreducible nondegenerate (i.e., not contained in a hyperplane) curve of degree $d\ge n$ in $\PP^n$ is bounded from above by
$$C(d,n)=\frac{1}{2}m((m-1)(n-1)+2e),\quad\text{where}\ d-1=m(n-1)+e,\ 0\le e<n-1,$$
(see \cite{Ca} or \cite[Page 57]{GH}).

\begin{theorem}\label{t-pecker}
For any $d\ge n\ge3$ and any $\delta\le C(d,n)$, there exists a real irreducible nondegenerate curve of degree $d$ in $\PP^n$ with $\delta$ real nodes as its only singularities.
\end{theorem}

For the proof, Pecker \cite{Pe1} constructs a suitable real plane rational curve with $C(d,n)$ real nodes in the affine plane, then maps it by
$$\psi(x,y)=(x,x^2,...,x^{n-k},y,yx,yx^2,...,yx^{k-1}),\quad k=n-\left[\frac{d-1}{m+1}\right],$$ to $\PP^n$ with the image of degree $d$. It is not difficult to see that prescribed nodes of the obtained curve can be smoothed out (cf. also \cite{Ta1}).

\medskip\noindent{\bf Deformations of plane curves singularities.} A local version of the problems discussed in this survey is the following local adjacency problem:

\smallskip{\it Given a reduced plane curve singular germ $(C,p)$, what collections of singularities can appear in its (versal) deformation?}

\smallskip The question on the existence of a global plane curve of a given degree $d$ with prescribed singularities can be considered as the above local deformation question for an ordinary $d$-fold singular point.

We show only two specific examples, both over the real field and both concerning the nodal deformations of arbitrary real plane curve singular points.

The first result is due to Pecker \cite{Pe2}. Recall that the maximal number of nodes appearing in a deformation of a plane curve singularity $(C,p)$ equals $\delta(C,p)$, which in case of an irreducible (i.e., unibranch) germ $(C,p)$ can be written as $\frac{1}{2}\mu(C,p)$.

\begin{theorem}\label{t=pecker} \index{deformation!real elliptic}
Given an irreducible real plane curve singularity $(C,p)$ and any nonnegative $a\le\frac{1}{2}\mu(C,p)$. Then there exists a real deformation of $(C,p)$, whose general member has $a$ real elliptic nodes as its only singularities.
\end{theorem}

Due to the openness of versality (see, for instance, \cite[Theorem I.1.15]{GLS-ISD}), given a deformation of a singularity $(C,p)$ with a singular general member, there exists a deformation of $(C,p)$ in which the singularities of that general member can independently be deformed in a prescribed way. That is, to prove the theorem it is enough to find a deformation realizing $a=\frac{1}{2}\mu(C,p)$ elliptic nodes. For the latter deformation, Pecker explicitly constructs a deformation of the parametrization of $(C,p)$.

An  in a sense opposite question is to find a deformation with the maximal possible number of hyperbolic nodes. Such deformations are called {\em morsifications}, and they carry out an important information on the topology of the singularity $(C,p)$ \cite{AC,GZ} (see also \cite{FPST} for the relation of morsifications to mutations of quivers). A'Campo and Gusein-Zade \cite{AC,GZ} proved the following claim.

\begin{theorem}\label{t-morsification}\index{morsification}
Every totally real plane curve singularity (i.e., a real plane curve singularity $(C,p)$ whose all local branches are real) possesses a morsification, and each morsification exhibits $\delta(C,p)$ hyperbolic nodes.
\end{theorem}

\noindent The proof is a combination of the blow-up induction and certain explicit formulas (the quintic shown in Figure \ref{fig-quintic}(a) represents, in fact, a morsification of the singularity $y^4-2x^5=0$).
\medskip

The question on the existence of morsifications for real singularities $(C,p)$ having complex conjugate local branches turns to be much harder. A partial answer to this question was suggested in \cite{LS}.

\subsection{Some questions and conjectures}

In principle, every time we presented a partial answer to a specific or general existence problem, we encourage the reader to improve or even complete the answer. However, several questions deserve a more detailed comment.

\medskip\noindent{\bf Cuspidal plane curves.} In Section \ref{ssec:2.2} we discussed one of the most challenging questions: what is the maximal number $k_{\max}(d)$ of ordinary cusps of a plane curve of degree $d$? Langer \cite{Lan} conjectures that the coefficient of $d^2$ in the right-hand side of (\ref{eq:langer}) is sharp. More precisely,

\begin{conjecture}\label{con1}
$$\lim_{d\to\infty}\sup\frac{k_{\max}(d)}{d^2}=\frac{125+\sqrt{73}}{432}.$$
\end{conjecture}

Concerning the maximal number $k_{\max,\R}(d)$ of real cusps on a real plane curve of degree $d$, the best existence result is Theorem \ref{t-nodecusp}. We conjecture that the coefficient of $d^2$ in (\ref{eq:realcusp}) is sharp, i.e.,

\begin{conjecture}\label{con2}
$$\lim_{d\to\infty}\sup\frac{k_{\max,\R}(d)}{d^2}=\frac{1}{4}.$$
\end{conjecture}

The following version of the problem was pointed by Vik. Kulikov. Choose an almost complex structure on the plane tamed by the standard symplectic structure.
\medskip

\smallskip\noindent{\it Question 1

What is the maximal number of cusps of a pseudo-holomorphic plane curve of degree $d$? Does there exists a cuspidal plane pseudoholomorphic curve of degree $d$ with the number of cusps breaking consequences of the Bogomolov-Miyaoka-Yau inequality (e.g., the Hirzebruch-Ivinskis bound (\ref{eq:Iv2}))?}

\smallskip
The latter question reflects the fact that there is no analogue of the Bogomolov-Miyaoka-Yau inequality for symplectic fourfolds.

\medskip\noindent{\bf Reducible equisingular families of plane curves.} In contrast to the sufficient $T$-smoothness conditions of equisingular families of plane curves, which were shown to be sharp (or close to sharp) in several important cases (see Section \ref{ssec:smoothirred}), the known examples of reducible equisingular families like in Theorem \ref{t-red} are very far from the available general sufficient irreducibility conditions, which consist of three inequalities (see \cite[Theorem 4.6.4]{GLS3})
\begin{align}\max_{1\le i\le r}\nu'(S_i)&\le\frac{2}{5}d-1,\nonumber\\
\sum_{i=1}^r(\nu'(S_i)+2)^2&<\frac{9}{10}d^2,\nonumber\\
\frac{25}{2}\cdot\#(\text{nodes})+18\cdot\#(\text{cusps})+\sum_{S_i\ne A_1,A_2}(\tau'(S_i)+2)^2&<d^2,\label{eq:irr}\end{align}
where the singularity invariants $\nu'$ and $\tau'$ are of the order of the Tjurina number $\tau$, and hence the coefficients assigned to the singularities in (\ref{eq:irr}) are of order $\tau^2$. For instance, an ordinary $m$-fold singular point (considered up to topological equivalence) enters the left-hand side of (\ref{eq:irr}) with the coefficient $\frac{1}{4}m^2(m+1)^2$ (see \cite[Corollary 4.6.7]{GLS3}), while in the series of reducible equisingular families of curves with ordinary singularities from Theorem \ref{t-red}(3), the ratio of $d^2$ to the number of ordinary $m$-fold singularities does not exceed $(m+1)^2$. This leaves completely open the following question
\medskip

\smallskip\noindent{\it Question 2

How sharp are the sufficient irreducibility conditions (\ref{eq:irr})?}
\smallskip

Another specific feature of the examples in Theorem \ref{t-red}, namely, the fact that the curves in different components of the equisingular family have the same fundamental group of the complement (i.e., form a so-called {\em anti-Zariski pair} \index{pair!anti-Zariski}) raises the following interesting question:
\medskip

\smallskip\noindent{\it Question 3

Can the curves in an anti-Zariski pair be transferred to each other by a homeomorphism of the plane onto itself?}

\medskip\noindent{\bf Sharpness of restrictions to curves with arbitrary singularities.} We have discussed above the sharpness of the known restrictions, notably, of Langer's bound in the case of curves with ordinary cusps. On the other hand, in Section \ref{ssec:3.2} we have seen that, for $A_n$ singularities,  with large Milnor number $n$, Hirano's examples (Theorem \ref{t-An} and Remark \ref{r-hirano}) have almost the same asymptotics as the spectral bound does. Beyond the range of simple or ordinary multiple singularities, the spectral bound and the genus and Pl\"ucker formulas are the only universal bounds applicable to arbitrary singularities, and the spectral bound is much stronger than the genus and Pl\"ucker bounds. So, it is natural to ask
\medskip

\smallskip\noindent{\it Question 4

For which singularity types (say, semiquasihomogeneous, irreducible, etc.) with large Milnor numbers is the spectral bound (asymptotically) sharp, or almost sharp?}

As said above, so far this is known to be true only for $A_n$ singularities.
\medskip

\medskip\noindent{\bf Gromov-Witten invariants of rational surfaces.}
Let $\PP^2_r$ be the plane blown up at $r>0$ generic points. For $r\le 9$, we know a complete answer about the existence of nodal curves of arbitrary genus in an arbitrary linear system on $\PP^2_r$ (see Theorem \ref{t-others}(2)).

If $r>9$, one can find in the literature only partial answers, see \cite[Theorem 5 and Corollary 3.1.7]{GLS2} or
\cite[Theorem 4.5.30 and Corollary 4.5.31]{GLS3}. For a divisor class $D\in\Pic(\PP^2_r)$, the expected dimension
of the moduli space ${\mathcal M}_{0,g}(\PP^2_r,D)$ of stable maps of (unmarked) curves of genus $g$ to $\PP^2_r$ representing the class $D$ equals (cf. \cite{FP})
$$-DK_{\PP^2_r}+g-1\ .$$
The following question arises
\medskip

\smallskip\noindent{\it Question 5

Suppose that $r>9$ and $D\in\Pic(\PP^2_r)$ satisfies the conditions $-DK_{\PP^2_r}>0$ and $D^2\ge-1$. Does there exist a nodal rational curve $C\in|D|$?}

\medskip
The restriction $D^2\ge-1$ comes from the fact that $\PP^2_r$ does not contain $(-k)$-curves with $k>1$. The above question is directly related to the non-vanishing of genus zero Gromov-Witten invariants of $\PP^2_r$: it is shown in \cite[Theorem 4.1 and Section 5.2]{GP} that these Gromov-Witten invariants do count rational curves in $|D|$ if either $-DK_{\PP^2_r}>1$, or $d\le10$, or some $d_i$ equals $1$ or $2$, where $D=dL-d_1E_1-...-d_rE_r$ ($L$ being the lift of a generic line in $\PP^2$, $E_1,...,E_r$ the exceptional divisors of the blowing up). We note also that, in view of the condition
$-DK_{\PP^2_r}>0$, an affirmative answer to Question 5 yields the existence of a nodal curve $C'\in|D|$ with any nonnegative number of nodes fewer than for the rational curve $C$; hence, the nonvanishing of the corresponding Gromov-Witten invariants of positive genus. Furthermore, Question 5 can be extended in the following way:
\medskip

\smallskip\noindent{\it Question 5'

Suppose that $r>9$, $g\ge0$, and $D\in\Pic(\PP^2_r)$ satisfies the conditions $-DK_{\PP^2_r}+g>0$ and $D^2>0$. Does there exist a curve $C\in|D|$ of genus $g$? What is the enumerative meaning of the corresponding genus $g$ Gromov-Witten invariants of $\PP^2_r$?}\\

{\bf Acknowledgement: }
The second author was supported by the Israel Science Foundation grant no. 501/18 and by the Bauer-Neuman Chair in Real and Complex Geometry.

\bigskip

{\footnotesize
\noindent Gert-Martin Greuel\\
Technische Universität Kaiserslautern, Germany, e-mail: gm.greuel@gmail.com\\
Eugenii Shustin\\
Tel Aviv University, Israel, e-mail: shustin@tauex.tau.ac.il
}
\printindex

\begin{thebibliography}{99.}
\bibitem{AM} Abhyankar, S. S., Moh, T. T.: 
      Embeddings of the line in the plane. J. Reine Angew. Math. {\bf 276}, 148--166 (1975)
\bibitem{AC} A'Campo, H.: Le groupe de monodromie des singularit\'es isol\'ees des courbes planes, I.
Math. Ann. {\bf 213}, 1--32 (1975)
\bibitem{AG} Akyol, A., Degtyarev, A.: Geography of irreducible plane sextics. Proc. Lond. Math. Soc. (3)
{\bf111}, no. 6, 1307--1337 (2015)
\bibitem{AH1} Alexander, J., Hirschowitz, A.: An asymptotic vanishing theorem for generic unions of multiple points. Invent. Math. {\bf 140}, no. 2, 303--325 (2000)
\bibitem{BK} Brieskorn, E., Kn\"orrer, H.: Plane algebraic curves.
Birkh\"auser-Verlag, Switzerland (1986)
\bibitem{BG} Bruce, J. W., Giblin, P. J.:
A stratification of the space of plane quartic curves.
Proc. London Math. Soc. {\bf 42}, 270--298 (1981)
\bibitem{Br} Brusotti, L.: Sulla ``piccola variazione" di una
curva piana algebrica reali. Rend. Rom. Ac. Lincei (5) {\bf
30}, 375--379 (1921)
\bibitem{CPS} Calabri, A., Paccagnan, D., Stagnaro, E.: Plane algebraic curves with many cusps with
an appendix by Eugenii Shustin. Annali di Matematica Pura ed Applicata (4) {\bf193}, no. 3, 909--921 (2014)
\bibitem{Ca} Castelnuovo, G.: Ricerche di geometria sulle curve algebriche. Atti Reale Accademia delle Scienze di Torino {\bf 24}, 346--373 (1889)
\bibitem{XiChen} Chen, Xi: Rational curves on K3 surfaces. J. Alg. Geom. {\bf 8}, 245--278 (1999)
\bibitem{XiChen1} Chen, Xi: Nodal Curves on K3 Surfaces. New York J. Math. {\bf 25}, 168--173 (2019)
\bibitem{CNL} Cassou-Nogues, P., Luengo, I.: On $A_k$ singularities on plane curves of fixed degree.
Preprint (2000)
\bibitem{CC} Chiantini, L., Ciliberto, C.: On the Severi varieties of surfaces in $\PP^3$.
J. Alg. Geom. {\bf8}, no. 1, 67--83 (1999)
\bibitem{DGPS} Decker, W., Greuel, G.-M., Pfister, G., Sch\"onemann, H.: Singular 4-1-1 -- A computeralgebra system for polynomial computations (2018). https://www.singular.uni-kl.de
\bibitem{De1} Degtyarev, A.:
Isotopic classification of complex plane projective curves of degree 5.
Algebra i Analiz {\bf 1}, no. 4, 78--101 (1989)
\bibitem{De2} Degtyarev, A.: Topology of algebraic curves. An approach via dessins d'enfants. Walter de Gruyter,
Berlin/Boston (2012)
\bibitem{DPW} Du Plessis, A. A., Wall, C. T. C.: Singular hypersurfaces, versality and Gorenstein algebras.
J. Algebraic Geom. {\bf9}, no. 2, 309--322 (2000)
\bibitem{FPST} Fomin, S., Pylyavskyy, P., Shustin, E., Thurston, D.:  Morsifications and mutations. Preprint at arXiv:1711.10598.
\bibitem{FP} Fulton, W., Pandharipande, R. Notes on stable maps and quantum cohomology. Algebraic geometry---Santa Cruz 1995, Proc. Sympos. Pure Math., 62, Part 2, Amer. Math. Soc., Providence, RI, pp. 45--96 (1997)
\bibitem{Gou} Gourevich, A. Geometry of obstructed families of curves. J. Pure Appl. Algebra {\bf210}, no. 3, 721--734 (2007)
\bibitem{GP} G\"ottsche, L., Pandharipande, R.: The quantum cohomology of blow-ups of $\PP^2$ and enumerative geometry. J. Differential Geom. {\bf48}, no. 1, 61--90 (1998)
\bibitem{GK} Greuel, G.-M., Karras, U.: Families of varieties with
prescribed singularities. Compos. Math. {\bf 69}, no. 1, 83--110 (1989)
\bibitem{GL} Greuel, G.-M., Lossen, C.:
Equianalytic and equisingular families of curves on surfaces.
Manuscr.\ Math. {\bf 91}, 323--342 (1996)
\bibitem{GLS0} Greuel, G.-M., Lossen, C., Shustin, E.: New asymptotics in the geometry of equisingularfamilies of curves. Int. Math. Res. Not. {\bf13}, 595--611 (1997)
\bibitem{GLS1} Greuel G.-M., Lossen, C., Shustin, E.:
Plane curves of minimal degree with prescribed singularities.
Invent. Math. {\bf 133}, no. 3, 539--580 (1998)
\bibitem{GLS2} Greuel, G.-M., Lossen, C., Shustin, E.:
Families of nodal curves on the blown-up projective plane.
Trans. Amer. Math. Soc. {\bf 350}, no. 1, 251--274 (1998)
\bibitem{GLS5} Greuel, G.-M., Lossen, C., Shustin, E.: Castelnuovo function, zero-dimensional schemes and singular plane curves. J. Algebr. Geom. {\bf9}, no. 4, 663--710 (2000)
\bibitem{GLS4} Greuel, G.-M., Lossen, C., Shustin, E.: The Variety of Plane Curves with Ordinary Singularities is not irreducible. Intern. Math. Res. Notices {\bf 11}, 542--550 (2001)
\bibitem{GLS-ISD} Greuel, G.-M., Lossen, C., Shustin, E.: Introduction to singularities and deformations. Springer, Berlin (2007)
\bibitem{GLS3} Greuel, G.-M., Lossen, C., Shustin, E.: Singular algebraic curves. Springer, Switzerland (2018)
\bibitem{GH} Griffiths, P., Harris, J.: Principles of algebraic geometry. John Wiley, NY (1978)
\bibitem{Gu} Gudkov, D. A.: On the curve of 5th order with 5 cusps. Funct. Anal. Appl.
{\bf16}, 201--202 (1982)
\bibitem{Gu1} Gudkov, D. A., Tai, M. L., Utkin, G. A.:
Complete classification of irreducible curves of the 4th order.
Mat. Sbornik {\bf 69}, no. 2, 222--256 (1966)
\bibitem{GZ} Gusein-Zade, S. M.: Dynkin diagrams for singularities of functions of two variables.
Func. Anal. Appl. {\bf 8}, 295--300 (1974)
\bibitem{GZN} Gusein-Zade, S. M., Nekhoroshev, N. N.: Singularities of type $A_k$ on simple curves of chosen degree. Funct. Anal. Appl. {\bf34}, no. 3, 214--215 (2000)
\bibitem{HJ} Harris, J.: On the Severi problem. Invent. Math. {\bf 84}, 445--461 (1985)
\bibitem{HA} Hirano, A.: Constructions of plane curves with cusps.
Saitama Math. J. {\bf 10}, 21--24 (1992)
\bibitem{Hir} Hironaka, H.: On the arithmetic genera and the effective genera of algebraic curves.
Mem. Coll. Sci., Univ. of Kyoto {\bf30}, 177--195 (1957)
\bibitem{HF} Hirzebruch, F.: Singularities of algebraic surfaces
and characteristic numbers. Contemp. Math. {\bf 58}, 141--155 (1986)
\bibitem{IS} Itenberg, I., Shustin, E.: Real algebraic curves with real cusps.
Amer. Math. Soc. Transl. (2) {\bf 173}, 97--109 (1996)
\bibitem{Iv} Ivinskis, K.: Normale Fl\"achen und die Miyaoka-Kobayashi Ungleichung.
Diplomarbeit, Bonn (1985)
\bibitem{Ker} Kerner, D.: Recombination formulas for the spectrum of curve singularities and some applications. Preprint at arXiv:1403.5789.
\bibitem{Koe} Koelman, R.~J.: Over de cusp. University of Leiden, Diplomarbeit (1986)
\bibitem{KP1} Koras, M., Palka, K.: The Coolidge-Nagata conjecture. Duke Math. J. {\bf 166}, no. 16, 3085--3145 (2017)
\bibitem{KP2} Koras, M., Palka, K.: Complex planar curves homeomorphic to a line have at most four singular points. Preprint at arXiv:1905.11376.
\bibitem{KW1} Korchagin, A. B., Weinberg, D. A.: The isotopy classification of affine quartic curves. Rocky Mountain J. Math. {\bf32}, no. 1, 255--347 (2002)
\bibitem{KW} Korchagin, A. B., Weinberg, D. A.: Quadric, cubic, and quartic cones.
Rocky Mountain J. Math. {\bf 35}, no. 5, 1627--1656 (2005)
\bibitem{Ku} Kulikov, Vik. S.: The generalized Chisini conjecture.
Proc. Steklov Inst. Math. {\bf241}, no. 2, 110--119 (2003)
\bibitem{Lan} Langer, A.: Logarithmic orbifold Euler numbers of surfaces with applications.
Proc. Lond. Math. Soc. (3) {\bf86}, 358--396 (2003)
\bibitem{Lef} Lefschetz, S.: On the existence of loci with given singularities. Trans. Amer. Math. Soc. {\bf 14}, 23--41 (1913)
\bibitem{LS} Leviant, P., Shustin, E.: Morsifications of real plane curve singularities.
J. Singularities {\bf 18}, 307--328 (2018)
\bibitem{Lu} Luengo, I.: The $\mu$-constant stratum is not smooth.
Invent. Math. {\bf 90}, 139--152 (1987)
\bibitem{Lu1} Luengo, I.:  On the existence of complete families of projective plane curves, which are obstructed.
J. London Math. Soc. {\bf 2-36}, no. 1, 33--43 (1987)
\bibitem{Mi}  Milnor, J.: Singular points of complex hypersurfaces. Princeton Univ. Press, Princeton (1968).
\bibitem{Miy} Miyaoka, Y.: The maximal number of quotient singularities on surfaces with given numerical invariants.
Math. Ann. {\bf268}, 159--171 (1984)
\bibitem{N1} Newton, I.: Analysis of the properties of cubic curves and their classification by species.
The mathematical papers of Isaac Newton (D.T. Whiteside, ed.), Cambridge Univ. Press, vol. 2, pp. 389--34 (1968)
\bibitem{N2} Newton, I.: Enumeratio linearum tertij ordinis. The mathematical papers of Isaac
Newton (D.T.  Whiteside,  ed.),  Cambridge  Univ. Press, vol. 7, pp. 565--645 (1976)
\bibitem{N3} Newton, I.: The final ``Geometri\ae\ libri duo".
The mathematical papers of Isaac Newton (D.T. Whiteside, ed.), Cambridge Univ. Press, vol. 7, pp. 402--469 (1976)
\bibitem{Nor} Nori, M.: Zariski conjecture and related problems. Ann. Sci. Ec. Norm. Sup.
{\bf4} (16), 305--344 (1983)
\bibitem{Ore} Orevkov, S. Yu.: Some examples of real algebraic and real pseudoholomorphic curves.
{\it Perspectives in Analysis, Geometry, and Topology}. Progress in Mathematics, vol. 296,
pp. 355–387, Birkh\"auser, Basel (2012)
\bibitem{Pe1} Pecker, D.: Simple constructions of algebraic curves with nodes.
Compos. Math. {\bf 87}, 1--4 (1993)
\bibitem{Pe3} Pecker, D.: Note sur la r\'ealit\'e des points doubles des courbes gauches.
C. R. Acad. Sci.Paris, S\'er. I {\bf324}, 807--812 (1997)
\bibitem{Pe2} Pecker, D.: Sur le th\`eor\'eme local de Harnack.
C. R. Acad. Sci. Paris, S\'er. I {\bf 326}, no. 5, 573--576 (1998)
\bibitem{Pe4} Pecker, D.: Sur la r\'ealit\'e des points doubles des courbes gauches. Ann.
Inst. Fourier (Grenoble) {\bf49}, 1439--1452 (1999)
\bibitem{Sa} Sakai, F.: Singularities of Plane Curves.
Geometry of Complex Projective Varieties
(Seminars and Conferences 9), Mediterranian Press, Rende, pp. 257--273 (1993)
\bibitem{Schu} Schulze, M.: A singular 3-1-6 library for computing invariants related to the Gauss-Maninsystem of an isolated hypersurface singularity (2012) (gmssing.lib.)
\bibitem{Segre} Segre, B.: Dei sistemi lineari tangenti ad un qualunque sistema di forme.
Atti Acad. naz. Lincei Rendiconti serie 5, {\bf 33},
182--185 (1924)
\bibitem{Segre2} Segre, B.: Esistenza e dimensione di sistemi continui di curve piane algebriche
con dati caraterri. Atti Acad. naz. Lincei Rendiconti serie 6, {\bf 10},
31--38 (1929)
\bibitem{Se} Severi, F.: Vorlesungen \"uber Algebraische Geometrie (Anhang F).
Leipzig, Teubner (1921)
\bibitem{Sh14} Shustin, E. I.: On invariants of algebraic curve singular points. Math. Notes of Acad. Sci. USSR {\bf34}, 962--963 (1983)
\bibitem{Sh2} Shustin, E.: Versal deformations in the space of plane curves of fixed degree.
Func. Anal. Appl. {\bf 21}, 82--84 (1987)
\bibitem{Sh4} Shustin, E.: On manifolds of singular algebraic curves. Selecta Math. Sov.
{\bf 10}, no.1, 27--37 (1991)
\bibitem{Sh6} Shustin, E.: Real plane algebraic curves with prescribed singularities. Topology
{\bf 32}, no. 4, 845--856 (1993)
\bibitem{Sh7} Shustin, E.: Smoothness and irreducibility of varieties of algebraic curves with nodes
and cusps. Bull. Soc. Math. France {\bf 122}, 235--253 (1994)
\bibitem{Sh12} Shustin, E.: Gluing of singular and critical points.
Topology {\bf 37}, no. 1, 195--217 (1998)
\bibitem{Sh2000} Shustin, E.: Lower deformations of isolated hypersurface singularities. St. Petersburg Math. J.
{\bf11}, no. 5, 883--908 (2000)
\bibitem{Sh13} Shustin, E.: Analytic order of singular and critical points.
Trans. Amer. Math. Soc. {\bf 356}, 953--985 (2004)
\bibitem{ST1} Shustin, E., Tyomkin, I.: Patchworking singular algebraic curves, I. Israel J. Math.
{\bf151}, 125--144 (2006)
\bibitem{ShW} Shustin, E., Westenberger, E.: Projective hypersurfaces with many singularities of pre-scribed types.
J. Lond. Math. Soc. (2) {\bf70}, no. 3, 609--624 (2004)
\bibitem{St0} Steenbrink, J. H. M.: Intersection form for quasihomogeneous singularities. Compos. Math.
{\bf 34}, no. 2, 211--223 (1977)
\bibitem{St} Steenbrink, J. H. M.: Semicontinuity of the singularity spectrum. Invent. Math.
{\bf 79}, 557--566 (1985)
\bibitem{Ta-1} Tannenbaum, A.: On the geometric genera of projective curves. Math. Ann. {\bf240},
no. 3, 213--221 (1979)
\bibitem{Ta1} Tannenbaum, A.: Families of algebraic curves with nodes.
Compos. Math. {\bf 41}, 107--126 (1980)
\bibitem{Ta2} Tannenbaum, A.: Families of curves with nodes on K3 surfaces.
Math. Ann. {\bf 260}, 239--253 (1982)
\bibitem{Ta0} Tannenbaum, A.: On the classical characteristic linear series of plane curves
with nodes and cuspidal points: two examples of Beniamino Segre.
Compositio Math. {\bf 51}, 169--183 (1984)
\bibitem{Te} Teissier, B.: The hunting of invariants in the geometry of discriminants.
Real and Complex Singularities (P.Holm, ed.), Oslo.
Sijthoff-Noordhoff Publ., Alphen aan den Rijn, pp. 565--677 (1977)
\bibitem{U1} Urabe, T.: On quartic surfaces and sextic curves with certain singularities.
Proc. Japan. Acad. Ser. A. {\bf 59}, 434--437 (1983)
\bibitem{U2} Urabe, T.: On quartic surfaces and sextic curves with singularities of type
$\widetilde E_8,\ T_{2,3,7},\ \widetilde E_{12}$.
Publ. Inst. Math. Sci. Kyoto Univ. {\bf 80}, 1185--1245 (1984)
\bibitem{U3} Urabe, T.: Singularities in a certain class of quartic surfaces and sextic curves
and Dynkin graphs. Proc. Vancouver Conf. Algebraic Geom., July 2-12, 1984. Providence, pp. 477--497 (1986)
\bibitem{Va} Varchenko, A. N.: Asymptotic of integrals and Hodge
structures. Modern Problems of Math., Vol. 22 (Itogi nauki
i tekhniki VINITI), Moscow, pp. 130--166 (1983)
\bibitem{Var1} Varchenko, A.~N.: On semicontinuity of the spectrum and an upper estimate for the number of singular points of a projective hypersurface. Sov. Math. Dokl. {\bf27}, 735--739 (1983);
    translation from Dokl. Akad. Nauk SSSR {\bf270}, no. 6, 1294--1297 (1983)
\bibitem{Wa1} Wahl, J.: Deformations of plane curves with nodes and cusps.
Amer. J. Math. {\bf 96}, 529--577 (1974)
\bibitem{Wa2} Wahl, J.: Equisingular deformations of plane algebroid curves.
Trans. Amer. Math. Soc. {\bf 193}, 143--170 (1974)
\bibitem{W1} Wall, C. T. C.: Geometry of quartic curves.
Math. Proc. Camb. Phil. Soc. {\bf 117}, 415--423 (1995)
\bibitem{W2} Wall, C. T. C.: Highly singular quintic curves.
Math. Proc. Camb. Phil. Soc. {\bf 119}, 257--277 (1996)
\bibitem{Wes1} Westenberger, E.: Existence of hypersurfaces with prescribed singularities. Commun. Algebra
{\bf31}, no. 1, 335--356 (2003)
\bibitem{Wes} Westenberger, E.: Real hypersurfaces with many simple singularities.
Rev. Mat. Complut. {\bf 18}, no. 2, 455--464 (2005)
\bibitem{Xu} Xu, G.:  Ample line bundles on smooth surfaces. J. Reine Angew. Math. {\bf469}, 199--209 (1995)
\bibitem{ZL} Zaidenberg, M. G., Lin, V. Ya.: An irreducible simply connected algebraic curve in $\CC^2$ is equivalent to a quasihomogeneous curve. Dokl. Akad. Nauk SSSR {\bf271}, no. 5, 1048--1052 (1983)
\bibitem{Za0} Zariski, O.: On the Problem of Existence of Algebraic Functions of Two Variables Possessing a Given
Branch Curve. Amer. J. Math. {\bf51}, no. 2, 305--328 (1929)
\bibitem{Za1} Zariski, O.: Studies in equisingularity I-III. Amer. J. Math. {\bf 87}, 507--536, (1965);
Amer. J. Math. {\bf 87}, 972--1006 (1965); Amer. J. Math. {\bf 90}, 961--1023 (1968)
\bibitem{Za2} Zariski, O.: Algebraic surfaces, 2nd edn. Springer,
Heidelberg (1971)

\end{thebibliography}
\end{document}